\documentclass[10pt]{article}

\usepackage{amsmath, amsthm, amssymb, enumerate, graphicx}
\usepackage{color}
\usepackage{datetime}
\usepackage{fancyhdr}
\usepackage{comment}

\chardef\coloryes=1 
\chardef\isitdraft=0

\ifnum\isitdraft=1 
\usepackage[color,notcite]{showkeys}
   \def\eqref#1{({\ref{#1}})}                

\definecolor{refkey}{rgb}{.3,0.3,0.3}

\else
  \definecolor{refkey}{rgb}{.8,0.8,0.1}
  \definecolor{labelkey}{rgb}{.9,0.6,0.1}

  \def\startnewsection#1#2{\section{#1}\label{#2}\setcounter{equation}{0}}   
\fi

\begin{document}
\def\Dg{{D'g}}
\def\ua{u^{\alpha}}
\def\intint{\int\!\!\!\!\int}
\def\intinttext{\int\!\!\!\int}
\def\intintint{\int\!\!\!\!\int\!\!\!\!\int}
\def\intintintint{\int\!\!\!\!\int\!\!\!\!\int\!\!\!\!\int}
\def\ques{{\cor \underline{??????}\cob}}
\def\nto#1{{\coC \footnote{\em \coC #1}}}
\def\fractext#1#2{{#1}/{#2}}
\def\fracsm#1#2{{\textstyle{\frac{#1}{#2}}}}   
\def\baru{U}
\def\nnonumber{}
\def\palpha{p_{\alpha}}
\def\valpha{v_{\alpha}}
\def\qalpha{q_{\alpha}}
\newcommand\Gampl{\Gamma_{\textup{pl}}}
\newcommand\wtil{\tilde{w}}
\newcommand\Wtil{\tilde{W}}
\newcommand\vtil{\tilde{v}}
\def\walpha{w_{\alpha}}
\def\falpha{f_{\alpha}}
\def\dalpha{d_{\alpha}}
\def\galpha{g_{\alpha}}
\def\halpha{h_{\alpha}}
\def\psialpha{\psi_{\alpha}}
\def\psibeta{\psi_{\beta}}
\def\betaalpha{\beta_{\alpha}}
\def\gammaalpha{\gamma_{\alpha}}
\def\Talpha{T}
\def\TTalpha{T_{\alpha}}
\def\TTalphak{T_{\alpha,k}}
\def\falphak{f^{k}_{\alpha}}
\def\R{\mathbb R}

\newcommand {\Dn}[1]{\frac{\partial #1  }{\partial N}}
\def\mm{m}

\def\cor{{}}
\def\cog{{}}
\def\cob{{}}
\def\coe{{}}
\def\coA{{}}
\def\coB{{}}
\def\coC{{}}
\def\coD{{}}
\def\coE{{}}
\def\coF{{}}

\ifnum\coloryes=1

  \definecolor{coloraaaa}{rgb}{0.1,0.2,0.8}
  \definecolor{colorbbbb}{rgb}{0.1,0.7,0.1}
  \definecolor{colorcccc}{rgb}{0.8,0.3,0.9}
  \definecolor{colordddd}{rgb}{0.0,.5,0.0}
  \definecolor{coloreeee}{rgb}{0.8,0.3,0.9}
  \definecolor{colorffff}{rgb}{0.8,0.3,0.9}
  \definecolor{colorgggg}{rgb}{0.5,0.0,0.4}
  \definecolor{colorhhhh}{rgb}{0.6,0.6,0.6}

 \def\cog{\color{colordddd}}
 \def\coy{\color{colorhhhh}}
 \def\cogray{\color{colorhhhh}}
 \def\cob{\color{black}}

 \def\coe{\color{blue}}
 \def\cor{\color{red}}

 \def\coA{\color{coloraaaa}}
 \def\coB{\color{colorbbbb}}
 \def\coC{\color{colorcccc}}
 \def\coD{\color{colordddd}}
 \def\coF{\color{colorffff}}
 \def\coG{\color{colorgggg}}

\fi
\ifnum\isitdraft=1
   \chardef\coloryes=1 
   \baselineskip=17pt
   \input macros.tex
   \def\blackdot{{\color{red}{\hskip-.0truecm\rule[-1mm]{4mm}{4mm}\hskip.2truecm}}\hskip-.3truecm}
   \def\bdot{{\coC {\hskip-.0truecm\rule[-1mm]{4mm}{4mm}\hskip.2truecm}}\hskip-.3truecm}
   \def\purpledot{{\coA{\rule[0mm]{4mm}{4mm}}\cob}}
   \def\pdot{\purpledot}
  \definecolor{labelkey}{rgb}{.5,0.1,0.1}
\else
   \baselineskip=15pt
   \def\blackdot{{\rule[-3mm]{8mm}{8mm}}}
   \def\purpledot{{\rule[-3mm]{8mm}{8mm}}}
   \def\pdot{}
\fi

\def\tdot{\fbox{\fbox{\bf\tiny\coe I'm here; \today \ \currenttime}}}
\def\nts#1{{\hbox{\bf ~#1~}}} 
\def\nts#1{{\cor\hbox{\bf ~#1~}}} 
\def\ntsf#1{\footnote{\hbox{\bf ~#1~}}} 
\def\ntsf#1{\footnote{\cor\hbox{\bf ~#1~}}} 
\def\bigline#1{~\\\hskip2truecm~~~~{#1}{#1}{#1}{#1}{#1}{#1}{#1}{#1}{#1}{#1}{#1}{#1}{#1}{#1}{#1}{#1}{#1}{#1}{#1}{#1}{#1}\\}
\def\biglineb{\bigline{$\downarrow\,$ $\downarrow\,$}}
\def\biglinem{\bigline{---}}
\def\biglinee{\bigline{$\uparrow\,$ $\uparrow\,$}}

\newcommand{\Dnu}{ \partial_{\nu}}
\renewcommand{\Gampl}{\Gamma_{pl}}
\def\mbar{{\overline M}}
\def\tilde{\widetilde}
\definecolor{darkgreen}{rgb}{0.,0.6,0.4}
\definecolor{darkcyan}{rgb}{0.,0.7,0.7}
\definecolor{darkpurple}{rgb}{0.4,0,0.4}
\def\ntAT#1{{\textcolor{blue}{\bf ~#1~}}} 
\def\ntBK#1{{\textcolor{darkgreen}{\bf ~#1~}}} 
\def\colAT#1{\textcolor{blue}{ ~#1~}} 
\def\colBK#1{\textcolor{darkgreen}{ ~#1~}} 
\def\chgd#1{\textcolor{darkpurple}{ ~#1~}} 
\def\rmv#1{\textcolor{green}{ ~#1~}} 
\def\damppl#1{#1} 
\def\Delpl{\Delta_{pl}}
\def\ff{\mathfrak{f}}
\def\gg{\mathfrak{g}}
\def\Ep{\mathcal{E}^p}
\def\Ew{\mathcal{E}^w}
\def\zeroT{0,T;}
\def\HminustwoGampl{H^2(\Gampl)^*}
\newtheorem{Theorem}{Theorem}[section]
\newtheorem{Corollary}[Theorem]{Corollary}
\newtheorem{Proposition}[Theorem]{Proposition}
\newtheorem{Lemma}[Theorem]{Lemma}
\newtheorem{Remark}[Theorem]{Remark}
\newtheorem{definition}{Definition}[section]
\def\theequation{\thesection.\arabic{equation}}
\def\endproof{\hfill$\Box$\\}
\def\square{\hfill$\Box$\\}
\def\comma{ {\rm ,\qquad{}} }            
\def\commaone{ {\rm ,\qquad{}} }         
\def\dist{\mathop{\rm dist}\nolimits}    
\def\sgn{\mathop{\rm sgn\,}\nolimits}    
\def\Tr{\mathop{\rm Tr}\nolimits}    
\def\div{\mathop{\rm div}\nolimits}    
\def\TT{R}
\def\supp{\mathop{\rm supp}\nolimits}    
\def\divtwo{\mathop{{\rm div}_2\,}\nolimits}    
\def\curl{\mathop{\rm curl}\nolimits}    
\def\dbar{\overline\partial}
\def\l{\langle}
\def\r{\rangle}
\def\plusdelta{+\delta}
\def\pd{+\delta}
\def\re{\mathop{\rm {\mathbb R}e}\nolimits}    
\def\indeq{\qquad{}\!\!\!\!}                     
\def\period{.}                           
\def\semicolon{\,;}                      
\newcommand{\cD}{\mathcal{D}}

\title{\mbox{Well-posedness of a Nonlinear Acoustics -- Structure Interaction Model}}

\author{ 
Barbara Kaltenbacher,
Amjad Tuffaha} 
\maketitle
\date{}
\bigskip

\bigskip
\indent Department of Mathematics\\
\indent University of Klagenfurt\\
\indent Klagenfurt, Austria\\
\indent e-mail: barbara.kaltenbacher@aau.at

\bigskip
\indent Department of Mathematics\\
\indent American University of Sharjah\\
\indent Sharjah, UAE\\
\indent e-mail: atufaha\char'100aus.edu

\bigskip
\begin{abstract}
We establish local-in-time and global in time well-posedness for small data, for a coupled system of nonlinear acoustic structure interactions. The model consists of the nonlinear Westervelt equation on a bounded domain with non homogeneous boundary conditions, coupled with a 4th order linear equation defined on a lower dimensional interface occupying part of the boundary of the domain, with transmission boundary conditions matching acoustic velocities and acoustic pressures. While the well-posedness of the Westervelt model has been well studied in the literature, there has been no works on the literature on the coupled structure acoustic interaction model involving the Westervelt equation. Another contribution of this work, is a novel variational weak formulation of the linearized system and a consideration of various boundary conditions.
\end{abstract}

\startnewsection{Introduction}{sec1}
In this paper, we study the interaction of a nonlinear acoustic fluid or gas with an elastic plate, a situation that is practically relevant, e.g., in the context of high intensity ultrasound; think, e.g., of a cavity enclosed by a thin wall \cite{BrownCox:2019} or a thin structure immersed in an ultrasound cleaning device \cite{Olson:1988}.

One of the classical models of nonlinear acoustics (and probably the most widely used one) is the Westervelt equation \cite{Westervelt63}, a quasilinear second order damped wave equation. For its analysis in the practically relevant case of a smooth bounded domain $\Omega$ 
we refer to, e.g., \cite{KL09Westervelt,KLV11_inhomDir,KL12_Neumann,MeyerWilke13,SimonettWilke2017}; the free space case $\Omega=\mathbb{R}^3$ is considered in, e.g., \cite{DekkersRozanova2019} and local in time results for vanishing damping can be found in, e.g., \cite{DoerflerGernerSchnaubelt:2016,DekkersRozanova2020,btozero}.
Advanced models such as 
Kuznetsov's equation \cite{Kuznetsov71},
the Jordan-Moore-Gibson-Thompson (JMGT) equation \cite{Jordan2014,MooreGibson1960,Thompson1972,RackeSaidHouari2020}
and the Blackstock-Crighton-Brunnhuber-Jordan (BCBJ) equation
\cite{Blackstock1963,Crighton1979,BrunnhuberJordan2016}
have been analyzed in, e.g., 
\cite{DekkersRozanova2019,DekkersRozanova2020,wellposednessBJ,KL12_Kuznetsov,KL12_Neumann,
KLM12_MooreGibson,KLP12_JordanMooreGibson,btozero,b2zeroJMGT,splittingModelsNonlinearAcoustics,
MeyerWilke13,MizohataUkai1993}.

For some recent work on the mathematical analysis  of acoustics-structure interaction problems we refer to, e.g., \cite{AvalosGeredeli2016,AvalosGeredeli2018,BecklinRammaha2021,Lasiecka2002,LasieckaRodrigues2021} and the references therein. Most of these models involve a linear or semi-linear wave equation coupled with linear and nonlinear plate models. Another  important aspect in the study of the mathematics of acoustics-structure interactions is control analysis, see for example \cite{BanksSmith1995, FahrooWang1999, Liu2022, YangYaoChen2018}.
However, to the best of the authors' knowledge, there are no results on well-posedness of nonlinear acoustics-structure interaction models involving the Westervelt equation.

The model we consider consists of the Westervelt equation defined on a bounded domain and coupled with a 4th order linear plate equation defined on a lower dimensional interface occupying a part of the boundary of the domain. The coupling is realized through the acoustic pressure which acts as a lifting force on the structural plate equation, in addition to a velocity matching condition  involving the normal pressure gradient. We also consider the possibility of mixed type non-homogeneous boundary conditions on the rest of the boundary such as Dirichlet, Neumann, and absorbing boundary conditions. Well-posedness of the Westervelt equations under these various mixed boundary conditions were considered in \cite{DekkersRozanovaTeplyaev} with vanishing boundary data.
Considering non-zero boundary data, we aim at laying the foundations for practically relevant control problems in future research.

Our main result in this paper, is firstly the existence of energy level weak solutions for the linearized model, and secondly the existence of local and global-in-time smooth solutions for the full nonlinear model. In particular, we provide a novel variational formulation of the coupled problem, that allows us to analyze the linearized problem (Proposition~\ref{prop:wellposed_lin_weak} and Corollary~\ref{cor:wellposed_lin}). Using this together with a fixed point argument, we establish local and global in time well-posedness for sufficiently small data (Theorem~\ref{theo:wellposed}).
Challenges in the analysis of the model arise due to the well-known potential degeneracy of the 
Westervelt equation (which it shares with basically all other models of nonlinear acoustics) in addition to the low regularity due to mixed boundary conditions. 
In particular, coupling to the plate provides less regularity than, e.g., prescribed  Dirichlet or Neumann boundary values would do. 
This requires a different approach as compared to the existing literature on the analysis of the Westervelt equation, e.g., \cite{KL09Westervelt,MeyerWilke2011,MizohataUkai1993}.
We note that in the mixed Dirichlet/Neumann boundary conditions, additional assumptions on the domain have to be imposed since classical elliptic regularity does not hold in general \cite{Plum1992, Savare1997}. This assumption allows us to control the $L^{\infty}$ norm of the pressure and hence avoid degeneracy without having to control the full Sobolev norm. 

The paper is organized as follows. Section~\ref{sec2} introduces the model and 
Section~\ref{sec:linmod} provides the analysis of its linearization.
In Section~\ref{sec:nl} we prove well-posedness of the nonlinear acoustics -- plate model and 
in Section~\ref{sec:outlook} we give a brief outlook on related further research questions.

\subsection{Notation}
We will use the function space 
$H^2_\Delta(\Omega):=\{u\in L^2(\Omega)\, : \, u\in L^2(\Omega), \, \Delta u\in L^2(\Omega)\}\supset H^2(\Omega)$

Moreover, we abbreviate by $C^\Omega_{\Gamma}$ the constant in the Poincar\'{e}-Friedrichs type inequality
\begin{equation}\label{PF}
\Vert v\Vert_{H^1(\Omega)}\leq C^\Omega_{\Gamma}
\Bigl(\Vert \nabla v \Vert_{L^2(\Omega)}^2+\Vert v \Vert_{L^2(\Gamma)}^2\Bigr)^{1/2} \text{ for all }v\in H^1(\Omega)
\end{equation}
for some regular boundary part $\Gamma\subseteq\partial\Omega$ with nonvanishing measure.

The norm of some continuous embedding $X(\Omega)\to Y(\Omega)$ for some Sobolev or Lebesgue spaces $X(\Omega)$, $X(\Omega)$ will be denoted by $C^\Omega_{X,Y}$.

We will sometimes drop the subscript in $\Vert\cdot\Vert_{L^2(\Omega)}$ (but not in the boundary norms, since this could cause confusion) and skip the trace operator whenever it is clear that we are taking some boundary trace of a function defined in $\Omega$.

The topological dual of a normed space $X$ will be denoted by  $X^{\star}$.
%

\startnewsection{Model}{sec2}

We consider a system of partial differential equations modeling the interaction between nonlinear acoustic waves and an isotropic, homogeneous plate. 
We take a domain $\Omega \subset \mathbb{R}^{3}$ 
with a $C^1$ boundary
$\partial \Omega$ consisting of four different parts $\Gampl$, $\Gamma_{a}$, $\Gamma_{D}$ and $\Gamma_{N}$, the plate coupling, absorbing, and excitation boundary parts, respectively.
Each of these boundary parts is assumed to be regular in the sense that it is either empty or has positive measure and the interfaces between them are assumed to be Lipschitz curves. An exemplary geometric setup is depicted in figure~\ref{fig:setup}.

\begin{figure}
\begin{center}
\includegraphics[width=0.6 \textwidth]{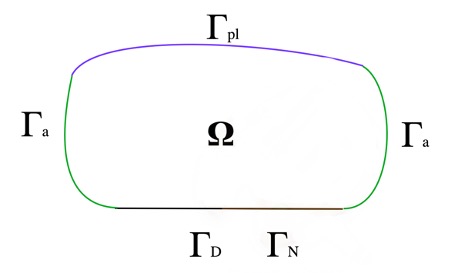}
\end{center}
\caption{Domain setup (example)\label{fig:setup}}
\end{figure}

The model consists of the Westervelt equation defined in the variable $p$ representing acoustic pressure 
 \begin{align}
\label{p}
(1 -  2k p) p_{tt}  - c^{2}\Delta p - b \Delta p_{t}  = 2k (p_{t})^{2} ~~\mbox{in}~~ \Omega \times [0,T]
\end{align}
and a $4$th order plate equation defined on $\Gampl$ in the mid surface displacement variable $w$, forced by the trace of the pressure variable $p$
\begin{align}
\label{pl}
\rho w_{tt} + \delta \Delpl^{2} w + \damppl{\beta(- \Delpl)^{\gamma} w_{t}} = \kappa p +h  ~~&\mbox{on}~~ \Gampl \times [0,T],
\end{align}
where $\Delpl$ is the Laplace-Beltrami operator on $\Gampl$, equipped with homogeneous Dirichlet boundary conditions, and $h$ is a given forcing function. \footnote{See, e.g., \cite{LasieckaRodrigues2021} for this type of coupling in a velocity potential formulation (rather than the pressure formulation we are using here).}  
Here, the constants $\rho>0$, $c>0$, $\delta>0$ represent the mass density, the speed of sound and the ratio of plate stiffness to plate thickness respectively, while $b>0$, $\beta>0$, and \damppl{
{$\gamma\geq0$}
} are damping parameters. The parameter $\kappa>0$ is a coupling coefficient and $k$ is the nonlinearity parameter of the acoustic medium.
\damppl{The operator $(- \Delpl)^{\gamma}$ is the spectrally defined power of the symmetric positive definite operator $(- \Delpl)$; for its definition, see, e.g., \cite[Section 2.1.1]{LischkeMeerschaert2020}.}

Accordingly, the boundary conditions we assume on the plate are of hinged type:
\begin{align}
\label{BCw}
 w = \Delpl w = 0 ~~ \mbox{on} ~~ \partial \Gampl \times [0,T].  
\end{align}

We also impose another boundary condition connecting the plate equation and the Westervelt equation which ensures continuity of inertial forces across the interface
\begin{align}
\label{BCp}
\Dnu{p} = - \rho w_{tt}  ~~&\mbox{on}~~ \Gampl \times [0,T].
\end{align}
Note that the pressure $p$ is related to the particle velocity $v$ through the (linearized) momentum balance
$ \rho_{fl} v_{t} =-  \nabla p$ where $\rho_{fl}$ is the mass density of the fluid. 

On the boundary component $\Gamma_{a}$, we impose the absorbing boundary conditions
\begin{align}
\label{BCa}
p_{t} +c\Dnu{p}=a  
~~&\mbox{on}~~ \Gamma_{a} \times [0,T]
\end{align}
where the typical value of the boundary function is $a\equiv0$ in order to avoid spurious reflections on the  boundary of the computational domain $\Omega$, see, e.g., the review
articles~\cite{GivoliBookChapter08,Hagstrom1999} and the references therein. 

Finally, we allow for mixed Dirichlet and/or Neumann boundary condition on the boundary components 
$\Gamma_{D}$ and $\Gamma_{N}$,
\begin{equation}
\label{BCe}
p=g_D  ~~\mbox{on}~~ \Gamma_{D} \times [0,T]
\end{equation} 
\begin{equation}
\label{BCeNeumann}
\Dnu{ p }=g_N  ~~\mbox{on}~~ \Gamma_{N} \times [0,T],
\end{equation} 
where $g_D$ and $g_N$ are given functions. 
Subsets with $g_D=0$ and $g_N=0$ correspond to sound-soft and sound-hard boundary parts, respectively; nonvanishing values of $g_D$ and $g_N$ represent an excitation control action; in case of Neumann conditions, e.g., by some piezoelectric transducer array in ultrasonics.

\begin{Remark}[damping]
While the damping term in the plate equation is not needed for part of the well-posedness analysis,
it is useful to achieve independence of estimates on the final time $T$ and to extract additional regularity of $p$ via the interface identity \eqref{BCptil}, see 
{Remark~\ref{rem:regularity}} below.

On the contrary, strong damping in the Westervelt equation is essential even for only local in time well-posedness, since due to the lack of spatial regularity resulting from the mixed boundary conditions, the techniques from, e.g., \cite{DoerflerGernerSchnaubelt:2016,btozero} would not be applicable here.
\end{Remark}

\section{Linearized Model}\label{sec:linmod}

\subsection{Reformulation, Linearization, and some Auxiliary Results}
We will reformulate the model in terms of $\wtil$ defined as 
\begin{align}
\wtil= w_{tt}
\end{align}
so that the boundary conditions on $\Gampl$ are written as 
\begin{align}
\label{wtil}
\rho\wtil_{tt} + \delta\Delpl^{2} \wtil +  \damppl{\beta(-\Delpl)^{\gamma} \wtil_{t}} = \kappa p_{tt} +\tilde{h}  ~~&\mbox{on}~~ \Gampl \times [0,T]
\end{align}
with $\tilde{h}=h_{tt}$ and
\begin{align}
\label{BCptil}
\Dnu{p} = - \rho \wtil  ~~&\mbox{in}~~ \Gampl \times [0,T].
\end{align}

We also start our analysis on a linear variable coefficient version of~\eqref{p}
 \begin{align}
\label{p_lin}
\alpha\, p_{tt}  - c^{2}\Delta p - b \Delta p_{t}  = f ~~\mbox{in}~~ \Omega \times [0,T]
\end{align}
(having in mind $\alpha = \alpha(t,x) =(1 -  2k p)$, $f=f(t,x)=2k (p_{t})^{2}$), thus altogether
\begin{equation}\label{pwlin}
\begin{aligned}
\alpha\, p_{tt}  - c^{2}\Delta p - b \Delta p_{t}  &= f  ~~\mbox{in}~~ \Omega \times [0,T] \\
 c\Dnu{ p}+p_{t} &= a  ~~\mbox{in}~~ \Gamma_{a} \times [0,T]  \\
 p&= g_D  ~~\mbox{on}~~ \Gamma_{D} \times [0,T] \\
\Dnu{ p}&=g_N ~~\mbox{on}~~ \Gamma_{N} \times [0,T] \\
\Dnu{ p} &= - \rho \wtil  ~~\mbox{on}~~ \Gampl \times [0,T] \\
\rho\wtil_{tt} + \delta\Delpl^{2} \wtil +  \damppl{\beta(-\Delpl)^{\gamma} \wtil_{t}} &= \kappa p_{tt} +\tilde{h}  ~~\mbox{on}~~ \Gampl \times [0,T] \\
\tilde{p}(0)&= p_{0}, ~~\tilde{p}_{t}(0)= p_{1}.
\end{aligned}
\end{equation}

We homogenize the Dirichlet and Neumann conditions on $\Gamma_{D}$, $\Gamma_{N}$ as well as the absorbing condition on $\Gamma_{a}$ in order to be able to consider $g_D=0$, $g_N=0$, $a=0$ in the following.
To this end, we decompose the solution $p$ into $\bar{p}$ and $\tilde{p}$ where $\bar{p}$ serves as an extension of the inhomogeneous boundary data $g_D$, $g_N$, $a$ to the interior by satisfying the initial boundary value problem IBVP
\begin{equation}\label{pbar}
\begin{aligned}
\alpha\, \bar{p}_{tt}  - c^{2}\Delta \bar{p} - b \Delta \bar{p}_{t}  &= 0 ~~\mbox{in}~~ \Omega \times [0,T] \\
c\Dnu{ \bar{p}} +\bar{p}_t&= a ~~\mbox{in}~~ \Gamma_{a} \times [0,T]  \\
 \bar{p}&= g_D  ~~\mbox{on}~~ \Gamma_{D} \times [0,T] \\
\Dnu{ \bar{p}}&=g_N ~~\mbox{on}~~ \Gamma_{N} \times [0,T] \\
\Dnu{ \bar{p}} &= 0  ~~\mbox{on}~~ \Gampl \times [0,T] \\
\bar{p}(0)&=0, ~~\bar{p}_{t}(0)=0.
\end{aligned}
\end{equation}
\begin{Lemma}\label{lem:extension}  
Under the assumptions 
$\alpha,\frac{1}{\alpha}\in L^\infty(\zeroT L^\infty(\Omega))$, 
$\alpha_t\in L^1(\zeroT L^\infty(\Omega))$ (with small enough norm), 
$\alpha_{tt}\in L^2(\zeroT L^3(\Omega))$, 
$\nabla\alpha\in L^2(\zeroT L^3(\Omega))$, 
$a_{ttt}\in L^2(0,t;L^2(\Gamma_a))$, 
$g_{D\,ttt}\in L^\infty(\zeroT H^{1/2}(\Gamma_D))$,
$g_{D\,tttt}\in L^1(\zeroT H^{1/2}(\Gamma_D))$, and
$g_{N\,ttt}\in W^{1,1}(\zeroT H^{-1/2}(\Gamma_N))$, 
there exists  a unique solution 
$\bar{p}\in W^{3,\infty}(\zeroT  H^1(\Omega)) \cap H^{3}(\zeroT  H^{2}_\Delta(\Omega))$ 
with $\textup{tr}_{\Gamma_a}\bar{p}_{tttt}\in L^{2}(\zeroT  L^{2}(\Gamma_a))$
to the IBVP \eqref{pbar}. 

In addition, for any $t\in(0,T)$, $\bar{p}$ satisfies the energy estimate
\begin{equation}\label{enest_eqp_pbar}
\Ep_1[\bar{p}](t)+\int_0^T\Ep_1[\bar{p}](s)\, ds \leq C\Bigl(\Ep_1[\bar{p}](0)
+ G[a,g_D,g_N] + G[a_{t},g_{D\,t},g_{N\,t}]
\Bigr),
\end{equation}
where $\Ep_1[\bar{p}]=\Ep[\bar{p}](t)+\Ep[\bar{p}_t](t)$, cf. \eqref{energy}, \eqref{E1}, and
\begin{equation}\label{G}
\begin{aligned}
&G[a,g_D,g_N]:= \Vert a_{t} \Vert_{L^2(0,t;L^2(\Gamma_a))}^2
+\Vert g_{D\,tt}\Vert_{L^2(\zeroT H^{1/2}(\Gamma_D))}^2
+\Vert g_{N\,t}\Vert_{W^{1,1}(0,t;H^{-1/2}(\Gamma_N))}^2
\end{aligned}
\end{equation}

The constant $C$ depends on the coefficients $b$, $c$, $\rho$, $\delta$, $\kappa$, 
$\Vert\frac{1}{\alpha}\Vert_{L^\infty(\zeroT L^\infty(\Omega))}$,  
$\Vert\alpha_t\Vert_{L^1(\zeroT L^\infty(\Omega))}$, but not directly on $T$.

\end{Lemma}
The proof of the above lemma goes analogously to the proof of e.g., 
\cite[Theorem 1]{BK-Peichl}; 
its most crucial steps are given in the appendix.

The remaining part $\tilde{p}$ together with $\tilde{w}$ then satisfies the equation
\begin{equation}\label{pwtil}
\begin{aligned}
\alpha\, \tilde{p}_{tt}  - c^{2}\Delta \tilde{p} - b \Delta \tilde{p}_{t}  &= f  ~~\mbox{in}~~ \Omega \times [0,T] \\
 c\Dnu{ \tilde{p}}+\tilde{p}_{t} &= 0  ~~\mbox{in}~~ \Gamma_{a} \times [0,T]  \\
 \tilde{p}&= 0  ~~\mbox{on}~~ \Gamma_{D} \times [0,T] \\
\Dnu{ \tilde{p}}&=0 ~~\mbox{on}~~ \Gamma_{N} \times [0,T] \\
\Dnu{ \tilde{p}} &= - \rho \wtil  ~~\mbox{on}~~ \Gampl \times [0,T] \\
\rho\wtil_{tt} + \delta\Delpl^{2} \wtil +  \damppl{\beta(-\Delpl)^{\gamma} \wtil_{t}} &= \kappa\tilde{p}_{tt}+ \kappa \bar{p}_{tt} +\tilde{h}  ~~\mbox{on}~~ \Gampl \times [0,T] \\
\tilde{p}(0)&= p_{0}, ~~\tilde{p}_{t}(0)= p_{1}.
\end{aligned}
\end{equation}

We thus continue by considering \eqref{pwtil} with 
$\tilde{h} + \kappa \bar{p}_{tt}$ replaced by $\tilde{h} \in W^{1,1}(\zeroT L^{2}(\Gampl)) $
while supressing the tilde notation for the $p$ variable, that is, we study \eqref{pwlin} with $g_D=0$, $g_N=0$, $a=0$. 
In fact, the high temporal regularity $\bar{p}\in W^{3,\infty}(\zeroT L^2(\Omega))$ in Lemma~\ref{lem:extension} is motivated by the need for $\tilde{h} + \kappa \bar{p}_{tt}\in W^{1,1}(\zeroT L^{2}(\Gampl)) $.

\medskip 
 
In the energy estimates below, we will see that the plate equation sometimes does not yield sufficient regularity of the trace $p\vert_{\Gampl}$ and we thus need to bound the values of $p$ on $\Gampl$ by relying on its interior regularity in $\Omega$ and a trace estimate.
While our energy estimates will allow us to bound $\Vert\Delta p\Vert_{L^2(\Omega)}+\Vert\nabla p\Vert_{L^2(\Omega)}$ we cannot expect full $H^2(\Omega)$ regularity from elliptic estimates -- again due to the fact that boundary conditions are too weak.
Thus, the following trace estimate (with $\vec{\varphi}=\nabla p$ in mind) will be useful.
\begin{Lemma}\label{lem:div}
For a function $\vec{\varphi} \in L^{2}(\Omega)^d$ with $\div{\vec{\varphi}} \in L^{2}(\Omega)$, the trace
$\vec{\varphi} \cdot \nu$ is a well defined element in $H^{-1/2}(\partial \Omega)$, the following inequality holds
\begin{align}\label{div}
\Vert \vec{\varphi}\cdot \nu  \Vert_{H^{-1/2}(\partial \Omega)} 
\leq C^{tr}_{-1/2} \Bigl(\Vert \vec{\varphi} \Vert_{L^{2}(\Omega)^d} + \Vert \div \vec{\varphi} \Vert_{L^{2}(\Omega)} \Bigr)
\end{align}
where $C^{tr}_{-1/2}$ is the norm of the inverse trace operator $H^{1/2}(\partial\Omega)\to H^1(\Omega)$.
\end{Lemma}
\begin{proof}
The estimate follows by a duality argument, using the identity 
$\int_{\partial\Omega} \vec{\varphi} \cdot \nu \, v\, dS =\int_\Omega \div(\vec{\varphi}\bar{v})\, dx$ 
for sufficiently smooth $\vec{\varphi}$, where $\bar{v}\in H^1(\Omega)$ is an extension of $v\in H^{1/2}(\partial\Omega)$, such that $\text{tr}_{\partial\Omega}\bar{v}=v$, and employing density.  
\end{proof}

For the same reason of $p(t)$ only being contained in the space $H^2_\Delta(\Omega)$ but not necessarily in $H^2(\Omega)$, when taking limits in the Galerkin approximation below, we will rely on the following auxiliary result.
\begin{Lemma}\label{lem:HDelta}
    The operator $\Delta:H^2_\Delta(\Omega)\to L^2(\Omega)$ is weakly closed. 
\end{Lemma}
\begin{proof}
    For any sequence $(p^n)_{n\in\mathbb{N}}$ such that (a) $p_m\rightharpoonup p$, (b) $\Delta p_m\rightharpoonup P$ in $L^2(\Omega)$, we have, for any $\phi\in C_0^\infty(\Omega)$
    \[
    \begin{aligned}
    &\int_\Omega\Delta p^n \, \phi \, dx =\int_\Omega p^n \Delta\phi\, dx \to \int_\Omega p \Delta\phi\, dx \text{ as }n\to\infty\ \text{ due to (a)}\\
    &\int_\Omega\Delta p^n \, \phi\, dx \to \int_\Omega P \phi\, dx \text{ as }n\to\infty\ \text{ due to (b)}
    \end{aligned}
    \]
    and therefore 
    \[
    \text{for all }\phi\in C_0^\infty(\Omega) \, : \ 
    \int_\Omega p \Delta\phi\, dx = \int_\Omega P \phi\, dx.
    \]
    Thus $\Delta p$ exists as a weak derivative and by uniqueness coincides with $P$ a.e., which also implies that  $\Delta p\in L^2(\Omega)$.
\end{proof}

\subsection{Weak Formulation of the Linearized System}

With the test and Ansatz space
\begin{equation}\label{eqn:V}
V=\{(\phi,\psi)\in H^2(\Omega)\times H^2(\Gampl)\, : \, \Dnu{\phi} =-\rho\psi\text{ on }\Gampl,\ 
\Dnu{\phi}=0\mbox{ on }\Gamma_{N}\}
\end{equation}
a variational formulation of \eqref{pwlin} can be given by
\begin{equation}\label{eqn:var}
\begin{aligned}
&(p(t),\wtil(t))\in V \text{ for a.e. }t\in (0,T) \text{ and }
\text{for all }(\phi,\psi)\in V\, : \\
&\int_\Omega \Bigl(\nabla p_{tt}\cdot\nabla \phi +\tfrac{c^2}{\alpha}\Delta p\,\Delta\phi+\tfrac{b}{\alpha}
\Delta p_t\,\Delta\phi\Bigr)\, dx
+\tfrac{\rho}{\kappa}\int_{\Gampl} \Bigl(\rho \wtil_{tt}\psi+\delta\Delpl\wtil\, \Delpl\psi  + \damppl{\beta(-\Delpl)^{\gamma} \wtil_{t} \psi}   \Bigr) \,dS
\\
&+\int_{\Gamma_{a}} c\partial_\nu p_t\, \partial_\nu\phi\, dS
\ = \int_\Omega \tfrac{f}{\alpha}\,(-\Delta\phi)\,dx + \tfrac{\rho}{\kappa}\int_{\Gampl} \tilde{h}\, \psi\, dS \\
&(p,p_t)(0)=(p_0,p_1), \quad (\wtil,\wtil_t)(0)=(\wtil_0,\wtil_1)
\end{aligned}
\end{equation}

Indeed, we have the following equivalence
  \begin{Lemma}\label{lem:equivalence}
If $(p,\wtil)\in \bigl(H^2(\zeroT H^1(\Omega))\cap H^1(\zeroT H^2_\Delta(\Omega))\bigr)$ $\times
\bigl(H^2(\zeroT \HminustwoGampl)\cap  H^1(\zeroT L^2(\Gampl))\cap  L^{\infty}(\zeroT H^2(\Gampl))\bigr)$
\footnote{Recall the notation
$H^2_\Delta(\Omega):=\{u\in L^2(\Omega)\, : \, u, \,\Delta u\in L^2(\Omega)\}\supset H^2(\Omega)$
}
solves \eqref{eqn:var} and the following compatibility conditions are satisfied:
\begin{equation}\label{eqn:compat}
p_1+c\partial_\nu p_0 =0 \text{ on }\Gamma_{a}, \quad
p_0=0
\text{ and }p_1=0
\text{ on }\Gamma_{D}
\end{equation}
Then $(p,\wtil)$ solves 
\eqref{pwlin} with $g_D=0$, $g_N=0$, $a=0$ (that is, \eqref{pwtil})
in an $L^2(\zeroT L^2(\Omega))\times L^2(\zeroT \HminustwoGampl)$ sense.

If additionally the following compatibility conditions hold on $\Gampl$
\begin{equation}\label{eqn:compat_wtil}
\begin{aligned}
&\wtil_0 = \kappa p_0-\delta\Delpl^2 w_0  \damppl{   - \beta(-\Delpl)^{\gamma} w_{1}}+\tilde{h}(0), \quad 
\wtil_1= \kappa p_1-\delta\Delpl^2 w_1   \damppl{- \beta(-\Delpl)^{\gamma} \wtil_{0}}+\tilde{h}_t(0), \\
&\partial_\nu p_0 = -\rho\wtil_0, \quad 
\partial_\nu p_1 = -\rho\wtil_1
\end{aligned}
\end{equation}
then with $w(t):=w_0+w_1t+\int_0^t\int_0^s \wtil(r)\, dr\, ds$, the pair 
$(p,w)$ solves \eqref{p_lin} and \eqref{pl} in an $L^2(\zeroT L^2(\Omega))\times L^2(\zeroT L^2(\Gampl))$ sense with boundary conditions \eqref{BCp}, \eqref{BCa}, and 
\eqref{BCe}.

\end{Lemma}

\begin{proof}
Integrating by parts and then using $\partial_\nu\phi =-\rho\psi$ on $\Gampl$ we obtain from \eqref{eqn:var}
\[
\begin{aligned}
0=&\int_\Omega \Bigl(p_{tt} -\tfrac{c^2}{\alpha}\Delta p - \tfrac{b}{\alpha}
\Delta p_t-\tfrac{f}{\alpha}\Bigr)(-\Delta\phi)\, dx +\int_{\partial\Omega}p_{tt}\partial_\nu \phi \,dS\\
&+\tfrac{\rho}{\kappa}\int_{\Gampl} \Bigl(\rho \wtil_{tt}+\delta\Delpl^2 \wtil 
\damppl{ + \beta(-\Delpl)^{\gamma} \wtil_{t}} 
-\tilde{h}\Bigr)\psi \,dS
+\int_{\Gamma_{a}} c\partial_\nu p_t\, \partial_\nu\phi\, dS 
\\
=&\int_\Omega \Bigl(p_{tt} -\tfrac{c^2}{\alpha}\Delta p -\tfrac{b}{\alpha}
\Delta p_t-\tfrac{f}{\alpha}\Bigr)(-\Delta\phi)\, dx 
+\int_{\Gamma_{a}} (c\partial_\nu p_t+p_{tt})\, \partial_\nu\phi\, dS 
+\int_{\Gamma_{D}} p_{tt}\partial_\nu\phi\, dS
\\
&+\tfrac{\rho}{\kappa}\int_{\Gampl} \Bigl(\rho \wtil_{tt}+\delta\Delpl^2 \wtil 
\damppl{ + \beta(-\Delpl)^{\gamma} \wtil_{t}} 
-\kappa p_{tt}-\tilde{h}\Bigr)\psi \,dS
\end{aligned}
\]
Now, if we choose test functions $(\phi,\psi) \in V$ such that $\psi=0$ and $\Dnu{\phi}=0$ on $\partial \Omega$, we obtain
\[
\begin{aligned}
0=&\int_\Omega \Bigl(p_{tt} -\tfrac{c^2}{\alpha}\Delta p - \tfrac{b}{\alpha}
\Delta p_t-\tfrac{f}{\alpha}\Bigr)(-\Delta\phi)\, dx.
\end{aligned}
\]
Note that given any function $q \in L^{2}(\Omega)$, one can find a function $\phi \in H^{2}(\Omega)$ such that $-\Delta\phi=q$ in $\Omega$ and $\Dnu{\phi}=0$ on the boundary. Hence, we may conclude that
\[
\begin{aligned}
0=&\int_\Omega \Bigl(p_{tt} -\tfrac{c^2}{\alpha}\Delta p - \tfrac{b}{\alpha}
\Delta p_t-\tfrac{f}{\alpha}\Bigr)\,q \,dx
\end{aligned}
\]
holds for all $q\in L^{2}(\Omega)$, and thus \eqref{p_lin} is satisfied in the $L^{2}$ sense.
On the other hand, if we choose the test functions $(\phi, \psi) \in V$ such that $\Delta \phi=0$ in $\Omega$, with $\Dnu{\phi}=\psi=0$ on $\Gampl$, and 
$\Dnu{\phi}=0$ on $\Gamma_{a}$, $\Dnu{\phi}= u$ on $\Gamma_{D}$ 
or vice versa $\Dnu{\phi}=0$ on $\Gamma_{D}$, $\Dnu{\phi}= u$ on $\Gamma_{a}$ respectively
for some arbitrary function $u \in H^{1/2}(\Gamma_{D})$ or $u \in H^{1/2}(\Gamma_{a})$, we obtain
\[
\begin{aligned}
\int_{\Gamma_{D}} p_{tt}\, u\, dS =0 \mbox{ for all }u \in  H^{1/2}(\Gamma_{D})
\end{aligned}
\]
and
\[
\begin{aligned}
\int_{\Gamma_{a}} (c\partial_\nu p_t+p_{tt})\,u\, dS=0  \mbox{ for all }u \in  H^{1/2}(\Gamma_{a}).
\end{aligned}
\]
Therefore, from the compatibility conditions \eqref{eqn:compat}, we conclude that \eqref{BCe} and \eqref{BCa} hold in $H^{-1/2}$ on $\Gamma_{D}$ and $\Gamma_{a}$ with $g_D=0$, $a=0$, respectively, .

To recover the plate equation \eqref{wtil}, we select test functions $(\phi,\psi) \in V$ with 
$\Delta \phi=0$ and $\Dnu{\phi}=0$ on $\Gamma_{a}$ and $\Gamma_{D}$, so that we obtain 
\[
\begin{aligned}
\int_{\Gampl} \Bigl(\rho \wtil_{tt}+\delta\Delpl^2 \wtil
\damppl{ + \beta(-\Delpl)^{\gamma} \wtil_{t}}
-\kappa p_{tt}   -\tilde{h}\Bigr)\psi \,dS =0
\end{aligned}
\]
for all $\psi \in H^{2}(\Gampl)$, and hence the pair $(\wtil, \wtil_{t})$ satisfies equation 
\eqref{wtil} in $\HminustwoGampl$. 

If we integrate \eqref{wtil} and \eqref{BCptil} twice in time, and utilize the additional compatibility conditions \eqref{eqn:compat_wtil}, we obtain that $(w,w_{t})$ satisfy \eqref{pl} and \eqref{BCp}. 

\end{proof}
\subsection{Well-Posedness of the Linearized System}

To prove well-posedness, as an intermediate result we will show existence of a solution to the following weaker variational formulation 
\begin{equation}\label{eqn:var_weak}
\begin{aligned}
&(p,\wtil)\in U\text{ and }
\text{for all }(q,\vtil)\in C^\infty([0,T],V), \, (q(T),\vtil(T))=(0,0)\, : \\
&\int_0^T\Bigl\{\int_\Omega \Bigl(-\nabla p_{t}\cdot\nabla q_{t} +\tfrac{c^2}{\alpha}\Delta p\,\Delta q +\tfrac{b}{\alpha}
\Delta p_t\,\Delta q\Bigr)\, dx\\
&\qquad +\tfrac{\rho}{\kappa}\int_{\Gampl}  \Bigl(-\rho\wtil_{t}\vtil_{t}+\delta\Delpl\wtil\, \Delpl\vtil \damppl{+ \beta(-\Delpl)^{\gamma} \wtil_{t} \vtil} \Bigr) \,dS
+\int_{\Gamma_{a}} c\partial_\nu p_t\, \partial_\nu q\, dS\Bigr\}\,dt
\\
&= \int_0^T\Bigl\{\int_\Omega \tfrac{f}{\alpha}\,(-\Delta q)\,dx 
+\tfrac{\rho}{\kappa}\int_{\Gampl} \tilde{h}\, \vtil\, dS 
+\int_\Omega \nabla p_1\cdot\nabla q(0)\, dx
+\tfrac{\rho}{\kappa}\int_{\Gampl} \rho\wtil_1 \vtil(0)\, dS \Bigr\}
\\
&p(0)=p_0, \quad \wtil(0)=\wtil_0
\end{aligned}
\end{equation}
in the solution space
\begin{equation}\label{U}
U=(W^{1,\infty}(\zeroT H^1(\Omega))\cap H^1(\zeroT H^2_\Delta(\Omega)))\times (W^{1,\infty}(\zeroT L^2(\Gampl))\cap L^\infty(\zeroT H^2(\Gampl)))
\end{equation}
induced by the energy
\begin{equation}\label{energy}
\begin{aligned}
\mathcal{E}[p,\wtil](t):=&
\Ep[p](t)+\Ew[\wtil](t) \text{ where }\\
\Ep[p](t):=&
\Vert p_{t}(t) \Vert_{H^1(\Omega)}^2
+\Vert \Delta p(t) \Vert_{L^2(\Omega)}^{2}
+ b \Vert \Delta p_{t} \Vert_{L^2(0,t;L^2(\Omega))}^{2}
+ \Vert \Dnu p_{t} \Vert_{L^2(0,t;L^{2}(\Gamma_{a}))}^{2}\\
\Ew[\wtil](t):=& 
\Vert \wtil_{t}(t) \Vert_{L^{2}(\Gampl)}^{2}
+ \Vert \Delpl \wtil(t) \Vert_{L^{2}(\Gampl)}^{2}
+ \beta \Vert (-\Delpl)^{\gamma/2} \wtil_{t} \Vert_{L^{2}(0,t;L^{2}(\Gampl))}^{2}
\end{aligned}
\end{equation}
where we track dependency on the damping parameters $b$, \damppl{$\beta$}.
Later on, also the higher energy
\begin{equation}\label{E1}
\mathcal{E}_1[p,\wtil]:=\mathcal{E}[p,\wtil]+\mathcal{E}[p_t,\wtil_t]
\end{equation}
will be used.

Accordingly, we define the energy of the initial data
\begin{align}
\mathcal{E}[p,\wtil](0) := 
{\Vert p_1 \Vert_{H^1(\Omega)}^2} 
+ \Vert \Delta p_0 \Vert_{L^{2}(\Omega)}^2 
+ \Vert \wtil_1 \Vert_{L^{2}(\Gampl)}^{2} + \Vert \Delpl \wtil_0 \Vert_{L^{2}(\Gampl)}^{2}
\end{align}
and the norms of the boundary/interior sources
\begin{align}\label{F}
F[f,\tilde{h}] :=  \Vert f \Vert_{L^1(\zeroT L^2(\Omega))}^{2} 
+\Vert\tilde{h}\Vert_{L^1(\zeroT L^2(\Gampl))}^2
\end{align}

\begin{Proposition}\label{prop:wellposed_lin_weak}
\begin{enumerate}
\item[(i)]
Suppose for some arbitrary $T\in(0,\infty]$, that  $\alpha$, $\frac{1}{\alpha}$ $\in L^\infty(\zeroT L^\infty(\Omega))$,  
$\alpha_t ~\in L^1(\zeroT L^\infty(\Omega))$ with $\bar{\alpha}:=\Vert\frac{\alpha_t}{\alpha}\Vert_{L^1(\zeroT L^\infty(\Omega))}<1/4$,  
$f\in L^1(\zeroT L^2(\Omega))$, 
$\tilde{h}\in L^1(\zeroT L^2(\Gampl))$,
$\nabla p_1 , \Delta p_0\in L^2(\Omega)$, and  $\tilde{w}_1, \Delpl\tilde{w}_0 \in L^2(\Gampl)$.
If $\Gamma_{D}=\emptyset$ we additionally assume 
{$\nabla\alpha\in L^\infty(\zeroT L^3(\Omega))$ with 
$\tilde{\alpha}_2=\Vert \alpha^{-3/2}\nabla\alpha \Vert_{L^2(0,t;L^3(\Omega))}$ and 
$\tilde{\alpha}_\infty=\Vert \alpha^{-3/2}\nabla\alpha \Vert_{L^\infty(0,t;L^3(\Omega))}$
small enough.}

Then there exists a solution $(p,\wtil)$ to \eqref{eqn:var_weak} in $U$ defined in \eqref{U}, which for any $t\in(0,T)$ satisfies the energy estimate
\begin{equation}\label{enest}
\mathcal{E}[p,\wtil](t)\leq 
 C \Bigl( \mathcal{E}[p,\wtil](0) + F[f,\tilde{h}]\Bigr)
\end{equation}
\item[(ii)]
Moreover,
\begin{equation}\label{enest_eqp_E}
\begin{aligned}
&\mathcal{E}[p,\wtil](t)+\int_0^t \mathcal{E}[p,\wtil](s)\, ds\\  
&\leq 
C \Bigl(\mathcal{E}[p,\wtil](0)
+ b \Vert \Delta p_0 \Vert_{L^2(\Omega)}^{2}
+\Vert\Dnu{p_0}\Vert_{L^2(\Gamma_a)}^2
\damppl{+\beta \Vert (-\Delpl)^{\gamma/2} \wtil_0 \Vert_{L^{2}(\Gampl)}^{2}}
+ F[f,\tilde{h}]  
\Bigr)
\end{aligned}
\end{equation}
\item[(iii)]
If, in addition to (i), 
$f_t\in L^1(\zeroT L^2(\Omega))$, 
$\tilde{h}_t\in L^1(\zeroT L^2(\Gampl))$,
$\nabla p_2$, $\Delta p_1$ $\in L^2(\Omega)$, and $\tilde{w}_2, \Delpl\tilde{w}_1$ $\in L^2(\Gampl)$, 
$\Dnu{p_2}=-\rho\wtil_2$ on $\Gampl$,
where
$p_2 := \tfrac{1}{\alpha(0)}\Bigl(c^2\Delta p_{0}+b\Delta p_1+f(0)\Bigr)$,
$\tilde{w}_2\:=\frac{1}{\rho}\Bigl(-\delta\Delpl^2\tilde{w}_0\damppl{-\beta(-\Delpl)^\gamma\tilde{w}_1}+\kappa p_2+\tilde{h}(0)\Bigr)$ then $(p,\wtil)$ satisfies the energy estimate
\begin{equation}\label{enest_prime}
\mathcal{E}_1[p,\wtil](t)\leq 
 C \Bigl( \mathcal{E}_1[p,\wtil](0) + F[f,\tilde{h}]  + F[f_t,\tilde{h}_t]\Bigr)
\end{equation}
{for $\mathcal{E}_1[p,\wtil](t):=\mathcal{E}[p,\wtil](t)+\mathcal{E}[p_t,\wtil_t](t)$.
}
The solution $(p,\wtil)$ to \eqref{pwtil} exists and is unique in 
\begin{equation}\label{Uprime}
U':=\{(p,\wtil)\in U\, : \, (p_t,\wtil_t)\in U\} \text{ with $U$ defined in \eqref{U}.}
\end{equation}
\item[(iv)] 
Moreover, under the conditions (i), (ii), (iii), 
\begin{equation}\label{enest_eqp_E1}
\begin{aligned}
&\mathcal{E}_1[p,\wtil](t)+\int_0^t \mathcal{E}_1[p,\wtil](s)\, ds\\  
&\leq 
C \Bigl(\mathcal{E}_1[p,\wtil](0)
+ b \Vert \Delta p_0 \Vert_{L^2(\Omega)}^{2}
+ b \Vert \Delta p_1 \Vert_{L^2(\Omega)}^{2}
+\Vert\Dnu{p_0}\Vert_{L^2(\Gamma_a)}^2
+\Vert\Dnu{p_1}\Vert_{L^2(\Gamma_a)}^2
\\
&\qquad\qquad\damppl{+\beta \Vert (-\Delpl)^{\gamma/2} \wtil_0 \Vert_{L^{2}(\Gampl)}^{2}}
\damppl{+\beta \Vert (-\Delpl)^{\gamma/2} \wtil_1 \Vert_{L^{2}(\Gampl)}^{2}}
+ F[f,\tilde{h}]  + F[f_t,\tilde{h}_t]
\Bigr)
\end{aligned}
\end{equation}
\end{enumerate}

\noindent The constants $C$ in (i), (ii), (iii) depend on the coefficients $b$, $c$, $\rho$, $\delta$, $\kappa$, 
$\Vert\frac{1}{\alpha}\Vert_{L^\infty(\zeroT L^\infty(\Omega))}$,  
$\Vert\alpha_t\Vert_{L^1(\zeroT L^\infty(\Omega))}$, but not directly on $T$.
\end{Proposition}
\begin{proof} 

\noindent{\bf \underline{Step 1} Galerkin approximation:} (note that this is still done on \eqref{eqn:var} rather than \eqref{eqn:var_weak})\\
We use an $L^2$ orthonormal eigensystem $(\phi_j,\lambda_j)_{j\in\mathbb{N}}$ of the negative Laplacian on $\Omega$ with homogeneous Neumann values on $\Gampl\cup\Gamma_{N}$
and homogeneous Dirichlet values on $\Gamma_{D}$
\[
-\Delta \phi_j=\lambda_j\phi_j\text{ in }\Omega, \qquad \Dnu{\phi}_j=0 \text{ on }\Gampl\cup\Gamma_{N}\,, \qquad\phi =0 \text{ on }~\Gamma_{D}
\]
and 
an $L^2_\rho$ orthonormal eigensystem $(\psi_i,\mu_i)_{i\in\mathbb{N}}$ of the negative Laplace Beltrami operator with Dirichlet boundary conditions on $\Gampl$, 
which we harmonically extend to functions $\tilde{\phi}_i$ on $\Omega$ such that
\[
-\Delta \tilde{\phi}_i=0\text{ in }\Omega, \quad \Dnu{\tilde{\phi}}_i=-\rho\psi_i \text{ on }\Gampl\,,\quad
\Dnu{\tilde{\phi}}_i=0 \text{ on }\Gamma_{a}\cup\Gamma_{N}\,, \quad\tilde{\phi}_i =0 \text{ on }~\Gamma_{D}.
\]
Our Galerkin spaces are then defined by $V^n:=\mbox{span}((\phi_1,0),\ldots,(\phi_n,0),\,(\tilde{\phi}_1,\psi_1),\ldots,(\tilde{\phi}_n,\psi_n))$.
For any fixed $n\in\mathbb{N}$ we make the ansatz
\[
p^n:= \sum_{j=1}^n p^n_j(t)\phi_j+\sum_{i=1}^n \wtil_i^n(t) \tilde{\phi}_i, \qquad
\wtil^n:= \sum_{i=1}^n \wtil_i^n(t) \psi_i
\]
so that in particular $\Dnu{p}^n=-\rho\wtil^n$ holds on $\Gampl$ and due to smoothness of the involved eigenfunctions, $(p^n,\wtil^n)\, \in V$.
Inserting into \eqref{eqn:var} and testing with $(\phi_\ell,0)$, $(\tilde{\phi}_k,\psi_k)$, we end up with the following ODE system for the coefficient vector functions $\underline{p}^n$, $\underline{\wtil}^n$:
\begin{eqnarray}
\label{ODEp}
&&\text{diag}(\lambda_1,\ldots,\lambda_n)(\underline{p}^n)'' 
+ b M^n(t) (\underline{p}^n)'
+ c^2 M^n(t) (\underline{p}^n)
+ D^n (\underline{\wtil}^n)'' + C^n (\underline{\wtil}^n)'(t)= f^{p,n}\\
\label{ODEwtil}
&&\rho(\underline{\wtil}^n)''
\damppl{+\beta \text{diag}(\mu_1^\gamma,\ldots,\mu_n^\gamma) (\underline{\wtil}^n)'}
+\delta\text{diag}(\mu_1^2,\ldots,\mu_n^2) (\underline{\wtil}^n)= f^{\wtil,n}
\end{eqnarray}
where  
\[
\begin{aligned}
&M^n_{j,\ell}(t)= \lambda_j\lambda_\ell  \int_\Omega\tfrac{1}{\alpha}\phi_j\,\phi_\ell\,dx
+ c\int_{\Gamma_{a}} \Dnu{\phi}_j\,\Dnu{\phi}_\ell\, dS \,, \quad
\\
&C^n_{i,\ell}=c \int_{\Gamma_{a}} \Dnu{\tilde{\phi}}_i\,\Dnu{\phi}_\ell\, dS \,, \quad
D^n_{i,\ell}=\rho\int_{\Gampl} \psi_i\,\phi_\ell\, dS \,, \\
& f^{p,n}_\ell(t)=\lambda_\ell\int_\Omega\tfrac{f}{\alpha}\phi_\ell\, dx, \qquad
f^{\wtil,n}_k(t)=\int_{\Gampl}\rho\, \tilde{h}\, \psi_k\, dS,
\end{aligned}
\]
due to the fact that
\[
\int_\Omega\nabla\tilde{\phi}_i\,\nabla\phi_\ell\,dx =\int_{\Gampl}\Dnu{\tilde{\phi}_i}\phi_\ell\,dS = - \rho\int_{\Gampl} \psi_i\,\phi_\ell\, dS, \quad
\int_\Omega\tfrac{b}{\alpha}\Delta\tilde{\phi}_i\,\Delta\phi_\ell\,dx =0, \quad
\int_\Omega\tfrac{c^2}{\alpha}\Delta\tilde{\phi}_i\,\Delta\phi_\ell\,dx =0. 
\]
{Note that $\alpha$ depends on space and time so $M^n$ is in general not a diagonal matrix but clearly positive definite.}
Existence of a solution \eqref{ODEp}, \eqref{ODEwtil} on $(0,T)$ follows from standard ODE results and the fact that {this ODE system is only triangularly coupled so that \eqref{ODEwtil} can be solved for $\underline{\wtil}^n$ independently of \eqref{ODEp} and then $\underline{\wtil}^n$ inserted into \eqref{ODEp}}.

\noindent{\bf \underline{Step 2} energy estimates:}\\
{(We here aim for $T$ independence of estimates in view of our goal of proving global in time well-posednenss.)}
\\
Replacing $\phi$ with $p^{n}_{t}$ and $\psi$ with $\frac{\kappa}{\rho}\wtil^n_t$ in \eqref{eqn:var}, abbreviating $\Vert\cdot\Vert=\Vert\cdot\Vert_{L^2(\Omega)}$ and using the identity 
\begin{equation}\label{iddt}
\frac{1}{2} \frac{d}{dt} \int_\Omega \frac{1}{\alpha} v^2\, dx = \int_\Omega \Bigl(\frac{1}{\alpha} v\, v_t - \frac{1}{2}\frac{\alpha_t}{\alpha^2} v^2\Bigr)\, dx
\end{equation}
we obtain 
\begin{equation}\label{enid}
\begin{aligned}
&\frac{1}{2} \frac{d}{dt} \Vert \nabla p^{n}_{t} \Vert^{2}
+\frac{1}{2} c^2 \frac{d}{dt} \left\Vert \frac{1}{\sqrt{\alpha}} \Delta p^{n} \right\Vert^{2}
+ b\, \left\Vert \frac{1}{\sqrt{\alpha}} \Delta p^{n}_{t} \right\Vert^{2}
+ c\, \Vert \Dnu p^{n}_{t} \Vert_{L^{2}(\Gamma_{a})}^{2}\\
&+ \frac{1}{2} \rho \frac{d}{dt} \Vert \wtil^{n}_{t} \Vert_{L^{2}(\Gampl)}^{2}
+ \frac{1}{2} \delta \frac{d}{dt} \Vert \Delpl \wtil^{n} \Vert_{L^{2}(\Gampl)}^{2}
\damppl{+\beta \Vert (-\Delpl)^{\gamma/2} \wtil^{n} \Vert_{L^{2}(\Gampl)}^{2}}
\\
&= \int_\Omega \Bigl( -\frac{c^2}{2}\frac{\alpha_t}{\alpha^2}(\Delta p^{n})^2 - \frac{f}{\alpha}\Delta p^{n}_{t}\Bigr)\, dx
+\int_{\Gampl} \tilde{h}\, \wtil^{n}_{t}\, dS 
\\
&\leq 
\frac{c^2}{2} \left\Vert \frac{\alpha_t}{\alpha} \right\Vert_{L^\infty(\Omega)}  \left\Vert \frac{1}{\sqrt{\alpha}} \Delta p^{n} \right\Vert^{2}
+ \frac{1}{2b} \left\Vert \frac{1}{\sqrt{\alpha}} f \right\Vert^{2}
+ \frac{b}{2} \left\Vert \frac{1}{\sqrt{\alpha}} \Delta p^{n}_{t} \right\Vert^{2}\\
&\qquad + \Vert \tilde{h} \Vert_{L^{2}(\Gampl)} \, \Vert \wtil^{n}_{t} \Vert_{L^{2}(\Gampl)}
\end{aligned}
\end{equation}

After integrating with respect to time and taking the supremum with respect to $t$ on both sides of \eqref{enid} we obtain
\begin{equation}\label{enest1}
\begin{aligned}
&\Vert \nabla p^{n}_{t} \Vert_{L^\infty(\zeroT L^{2}(\Omega))}^2
+c^2(1-\bar{\alpha}) \left\Vert \frac{1}{\sqrt{\alpha}} \Delta p^{n} \right\Vert_{L^\infty(\zeroT L^2(\Omega))}^{2}
\hspace*{-0.5cm}+ b\, \left\Vert \frac{1}{\sqrt{\alpha}} \Delta p^{n}_{t} \right\Vert_{L^2(\zeroT L^2(\Omega))}^{2}
\hspace*{-0.5cm}+ 2c\, \Vert \Dnu p^{n}_{t} \Vert_{L^2(\zeroT L^{2}(\Gamma_{a}))}^{2}\\
&+\frac{\rho}{2} \Vert \wtil^{n}_{t} \Vert_{L^\infty(\zeroT L^{2}(\Gampl))}^{2}
+ \delta \Vert \Delpl \wtil^{n} \Vert_{L^\infty(\zeroT L^{2}(\Gampl))}^{2}
\damppl{+ \beta \Vert (-\Delpl)^{\gamma/2} \wtil_{t} \Vert_{L^{2}(\zeroT L^{2}(\Gampl))}^{2}}
\\
&\leq 
\Vert \nabla p_1 \Vert^2
+c^2\left\Vert \frac{1}{\sqrt{\alpha(0)}} \Delta p_0 \right\Vert^2
+ \rho \Vert \wtil_1 \Vert_{L^{2}(\Gampl)}^{2}
+ \delta \Vert \Delpl \wtil_0 \Vert_{L^{2}(\Gampl)}^{2}
\\
&\qquad+\frac{1}{b} \left\Vert \frac{1}{\sqrt{\alpha}} f \right\Vert_{L^1(\zeroT L^2(\Omega))}^{2}
+ \frac{2}{\rho} \Vert\tilde{h}\Vert_{L^1(\zeroT L^2(\Gampl))}^2
\end{aligned}
\end{equation}
where we have used Young's inequality in  $ \int_0^t\Vert \tilde{h} \Vert_{L^{2}(\Gampl)} \, \Vert \wtil^{n}_{t} \Vert_{L^{2}(\Gampl)}\, ds
\leq \Vert\tilde{h}\Vert_{L^1(\zeroT L^2(\Gampl))} \Vert\wtil^{n}_{t}\Vert_{L^\infty(\zeroT L^2(\Gampl))}
\leq \frac{1}{\rho} \Vert\tilde{h}\Vert_{L^1(\zeroT L^2(\Gampl))}^2 +\frac{\rho}{4}\Vert\wtil^{n}_{t}\Vert^{2}_{L^\infty(\zeroT L^2(\Gampl))}$.

\medskip

In case of $\Gamma_{D}\not=\emptyset$, {this together with the Poincar\'{e}-Friedrichs inequality {\eqref{PF}} already yields the energy estimate \eqref{enest} for $(p^{n},\wtil^{n})$ in place of $(p,\wtil)$}. Otherwise we need to bound the $L^2(\Omega)$ norm of $p^{n}_t(t)$ separately in order to arrive at its full $H^1(\Omega)$ norm.

\medskip

\noindent
{\bf \underline{Step 2.b} control of $L^{2}(\Omega)$ norm in case $\Gamma_{D}=\emptyset$:}

Taking $(\phi, \psi) \in V $ in \eqref{eqn:var} such that 
$-\Delta \phi= p^{n}_t$, 
$\Dnu{\phi}=\begin{cases}0 &\text{if}~~ B_ap^n_t=0\\
-\frac{B_ap^n_t}{|B_ap^n_t|^2}\frac{1}{\alpha}(c^2 B_ap^n +b B_ap^n_t)&\text{ otherwise }\end{cases}$ 
with $B_a v:=c\Dnu{v}+ v_t$, 
$\Dnu{\phi}=0$ on $
\Gamma_{N}\cup\Gampl$ 
and accordingly $\psi=0$, we obtain 
{
\[
\begin{aligned}
&\frac{1}{2} \frac{d}{dt} \Vert p^{n}_{t} \Vert^{2}
+\frac{1}{2} c^2 \Bigl(\frac{d}{dt} \left\Vert \frac{1}{\sqrt{\alpha}} \nabla p^{n} \right\Vert^{2}
	+\int_\Omega \frac{\alpha_t}{\alpha^2}|\nabla p^{n}|^2\, dx\Bigr)
+ b\, \left\Vert \frac{1}{\sqrt{\alpha}} \nabla p^{n}_{t} \right\Vert^{2}
\\
&= \int_\Omega \Bigl(\frac{f}{\alpha} p^{n}_t+\frac{1}{\alpha^2}p^n_t\nabla\alpha\cdot(c^2\nabla p^n+b\nabla p^n_t)\Bigr)\, dx
-\int_{\partial\Omega} \Bigl(B_a p^n_t\, \Dnu{\phi}-(c^2\Dnu{p^n}+b\Dnu{p^n_t})\frac{1}{\alpha}p^n_t\Bigr)\, dS, 
\end{aligned}
\]
where 
\[
\int_{\Gamma_a} \Bigl(B_a p^n_t\, \Dnu{\phi}-(c^2\Dnu{p^n}+b\Dnu{p^n_t})\frac{1}{\alpha}p^n_t\Bigr)\, dS
= c\int_{\Gamma_a}\frac{1}{\alpha} p_t^2\, dS+ \frac{b}{c}\int_{\Gamma_a}\frac{1}{\alpha} \frac{d}{dt}[p_t^2]\, dS.
\]
Integrating in time yields
}
\begin{equation}
\begin{aligned}
\label{L2}
& \frac{1}{2} \Vert p^{n}_{t}(t) \Vert^{2} 
+ \frac{c^2}{2} \Bigl(\left\Vert \frac{1}{\sqrt{\alpha}} \nabla p^{n}(t) \right\Vert^{2}
+ \int_0^t\int_\Omega \frac{\alpha_t}{\alpha^2}|\nabla p^{n}|^2\, dx\, ds\Bigr)
+\int_{0}^{t} b  \left\Vert \frac{1}{\sqrt{\alpha}} \nabla p^{n}_t \right\Vert^{2} \,dx \,ds\\
& 
+c\int_0^t\left\Vert \frac{1}{\sqrt{\alpha}} p^{n}_t\right\Vert^{2}_{L^2(\Gamma_a)}\, ds
+\frac{b}{c} \Bigl(\left\Vert \frac{1}{\sqrt{\alpha}} p^{n}_t(t)\right\Vert^{2}_{L^2(\Gamma_a)}
+ \int_0^t\int_{\Gamma_a} \frac{\alpha_t}{\alpha^2}|p^{n}|^2\, ds\Bigr) \\
& =
\frac{1}{2} \Vert p^{n}_1 \Vert^{2}
+\frac{c^2}{2} \left\Vert \frac{1}{\sqrt{\alpha}} \nabla p^{n}_0 \right\Vert
+\frac{b}{c} \left\Vert \frac{1}{\sqrt{\alpha}} p^{n}_1\right\Vert^{2}_{L^2(\Gamma_a)}
+\underbrace{\int_0^t\int_\Omega \Bigl(\frac{f}{\alpha} p^{n}_t+\frac{1}{\alpha^2}p^n_t\nabla\alpha\cdot(c^2\nabla p^n+b\nabla p^n_t)\Bigr)\, dx\, ds}_{(0)}\\
&\quad
-c^2\rho\underbrace{{\int_0^t}\int_{\Gampl}\frac{1}{\alpha}{\wtil^n} p^n_t\, dS\, ds}_{(I)}
-b\rho\underbrace{{\int_0^t}\int_{\Gampl}\frac{1}{\alpha}{\wtil^n_t} p^n_t\, dS\, ds}_{(II)}.
\end{aligned}
\end{equation}
We estimate 
\[
\begin{aligned}
(O)\leq &
\Vert \tfrac{f}{\alpha} \Vert_{L^1(0,t;L^2(\Omega))}
\Vert p^{n}_t \Vert_{L^\infty(0,t;L^2(\Omega))}
+c^2 \Vert \tfrac{1}{\sqrt{\alpha}} \nabla p^{n} \Vert_{L^\infty(0,t;L^2(\Omega))}
\Vert p^{n}_t \Vert_{L^2(0,t;L^6(\Omega))}
\Vert \alpha^{-3/2}\nabla\alpha \Vert_{L^2(0,t;L^3(\Omega))}\\
&+b \Vert \tfrac{1}{\sqrt{\alpha}} \nabla p^{n}_t \Vert_{L^2(0,t;L^2(\Omega))}
\Vert p^{n}_t \Vert_{L^2(0,t;L^6(\Omega))}
\Vert \alpha^{-3/2}\nabla\alpha \Vert_{L^\infty(0,t;L^3(\Omega))}
\end{aligned}
\]
where due to \eqref{PF} 
\[
\Vert p^{n}_t \Vert_{L^2(0,t;L^6(\Omega))}
\leq C_{H^1,L^6}^\Omega \Vert p^{n}_t \Vert_{L^2(0,t;H^1(\Omega))}
\leq C_{H^1,L^6}^\Omega C^\Omega_{\Gamma_a} 
\Bigl(\Vert \nabla p^{n}_t \Vert_{L^2(0,t;L^2(\Omega))}^2+\Vert p^{n}_t \Vert_{L^2(0,t;L^2(\Gamma_a))}^2\Bigr)^{1/2}\,.
\]
The terms (I) and (II) can be estimated by
\[
\begin{aligned}
(I)\leq&\frac{c^2\rho}{\underline{\alpha}}\Vert \wtil^n\Vert_{L^2(0,t;H^{1/2}(\Gampl))} \Vert p^n_t\Vert_{L^2(0,t;H^{-1/2}(\Gampl))} \\
(II)\leq&\frac{b\rho}{\underline{\alpha}}\Vert \wtil^n_t\Vert_{L^2(0,t;H^{1/2}(\Gampl))} \Vert p^n_t\Vert_{L^2(0,t;H^{-1/2}(\Gampl))} 
\end{aligned}
\]
where the terms containing ${\wtil^n}$ can be estimated by the 8th and 9th terms in \eqref{enest_eqp_n} and 
\[
\begin{aligned}
    \Vert p^n_t\Vert_{L^2(0,t;H^{-1/2}(\Gampl))}
&\leq C^{\Gampl}_{H^{1/2},H^{-1/2}} C^{tr}_{1/2}\Vert p^n_t\Vert_{L^2(0,t;H^1(\Omega){)}}\\
&\leq C^{\Gampl}_{H^{1/2},H^{-1/2}} C^{tr}_{1/2} C^\Omega_{\Gamma_a} 
\Bigl(\Vert \nabla p^{n}_t \Vert_{L^2(0,t;L^2(\Omega))}^2+\Vert p^{n}_t \Vert_{L^2(0,t;L^2(\Gamma_a))}^2\Bigr)^{1/2}\,.
\end{aligned}
\]

After using Young's inequality, this provides us with the lower order energy estimate
\begin{equation}\label{enest0}
\begin{aligned}
&\frac12\Vert p^{n}_{t} \Vert_{L^\infty(\zeroT L^{2}(\Omega))}^2
+c^2(1-\bar{\alpha}-\tilde{\alpha}_2) \left\Vert \frac{1}{\sqrt{\alpha}} \nabla p^{n} \right\Vert_{L^\infty(\zeroT L^2(\Omega))}^{2}\\
&\ + \Bigl(\frac{\min\{b(1-\tilde{\alpha}_\infty),c\}}{\Vert\alpha\Vert_{L^\infty(\Omega)}(C^\Omega_{\Gamma_a})^2} - (c^2 \tilde{\alpha}_2 + b\tilde{\alpha}_\infty)(C^\Omega_{H^1,L^6})^2\Bigr)
\Vert p^{n}_{t} \Vert_{L^2(\zeroT H^1(\Omega))}^{2}
+ \tfrac{b}{c} (1-\bar{\alpha})\left\Vert \frac{1}{\sqrt{\alpha}}  p^{n}_{t} \right\Vert_{L^2(\zeroT L^2(\Gamma_a))}^{2}
\\
&\leq 
\Vert p_1 \Vert^2
+c^2\left\Vert \frac{1}{\sqrt{\alpha(0)}} \nabla p_0 \right\Vert^2
+\frac12\left\Vert \frac{1}{\sqrt{\alpha}} f \right\Vert_{L^1(\zeroT L^2(\Omega))}^{2}
+ \tilde{C}
(\mathcal{E}[p,\wtil](0) +F[f,\tilde{h}])
\end{aligned}
\end{equation}
where 
$\tilde{\alpha}_2=\Vert \alpha^{-3/2}\nabla\alpha \Vert_{L^2(0,t;L^3(\Omega))}$,
$\tilde{\alpha}_\infty=\Vert \alpha^{-3/2}\nabla\alpha \Vert_{L^\infty(0,t;L^3(\Omega))}$.

\medskip

\noindent
{\bf \underline{Step 2.c} equipartition of energy:} 
To show (ii), we substitute $\phi=p^{n}$ and $\psi= \frac{\kappa}{\rho} \wtil^n$ in \eqref{eqn:var} and use \eqref{iddt}, which yields the equipartition of energy identity
\[
\begin{aligned}
& - \Vert \nabla p^{n}_{t} \Vert^{2} + \frac{d}{dt}\int_\Omega \nabla p^{n}_{t}\cdot\nabla p^{n}\, dx
+ c^2  \left\Vert \frac{1}{\sqrt{\alpha}} \Delta p^{n} \right\Vert^{2}
+ \frac{1}{2} b\, \frac{d}{dt} \left\Vert \frac{1}{\sqrt{\alpha}} \Delta p^{n} \right\Vert^{2}
+ \frac{1}{2} c\, \frac{d}{dt} \Vert \Dnu p^{n} \Vert_{L^{2}(\Gamma_{a})}^{2}\\
& - \rho \Vert \wtil^{n}_{t} \Vert_{L^{2}(\Gampl)}^{2} + \rho \frac{d}{dt} \int_{\Gampl} \wtil^{n}_{t} \wtil^{n} \, dS 
+  \delta  \Vert \Delpl \wtil^{n} \Vert_{L^{2}(\Gampl)}^{2}
\damppl{+ \frac12 \beta\frac{d}{dt}\Vert (-\Delpl)^{\gamma/2} \wtil \Vert_{L^{2}(\Gampl)}^{2}}
\\
&= \int_\Omega \Bigl( -b\frac{\alpha_t}{{2 }\alpha^2}(\Delta p^{n})^2 - \frac{f}{\alpha}\Delta p^{n}\Bigr)\, dx
+ \int_{\Gampl} \, \tilde{h}\, \wtil^{n}\, dS 
\end{aligned}
\]
and, similarly to above, 
{
using $\int_{\Gampl} \, \tilde{h}\, \wtil^{n}\, dS$ $\leq 
\Vert \tilde{h}\Vert_{\HminustwoGampl}\Vert\wtil^{n}\Vert_{H^2(\Gampl)}$ \\
$\leq \Vert\tilde{h}\Vert_{\HminustwoGampl}
\Vert\Delpl\wtil^{n}\Vert_{L^2(\Gampl)}
\Vert(-\Delpl)^{-1}\Vert_{L^2(\Omega)\to H^2(\Omega)} $,
} 
the estimate
\begin{equation}\label{eqpid}
\begin{aligned}
&c^2 \left\Vert \frac{1}{\sqrt{\alpha}} \Delta p^{n} \right\Vert_{L^2(\zeroT L^2(\Omega))}^{2}
+ \frac{b}{4} (1-2\bar{\alpha}) \, \left\Vert \frac{1}{\sqrt{\alpha}} \Delta p^{n} \right\Vert_{L^\infty(\zeroT L^2(\Omega))}^{2}
+ \frac{c}{2} \, \Vert \Dnu p^{n} \Vert_{L^\infty(\zeroT L^{2}(\Gamma_{a}))}^{2}\\
&+ \frac{\delta}{2} \Vert \Delpl \wtil^{n} \Vert_{L^2(\zeroT L^{2}(\Gampl))}^{2}
\damppl{+\frac12\beta \Vert (-\Delpl)^{\gamma/2} \wtil^{n} \Vert_{L^\infty(\zeroT L^{2}(\Gampl))}^{2}}
\\
&\leq
\underbrace{ \Vert \nabla p^{n}_{t} \Vert_{L^2(\zeroT L^{2}(\Omega))}^2 }_{{I}}
+ \int_\Omega \nabla p^{n}_1\cdot\nabla p^{n}_0\, dx
+ \sup_{t\in(0,T)}\int_\Omega -\nabla p^{n}_{t}\cdot\nabla p^{n}\, dx\\ 
&+\underbrace{ \rho \Vert \wtil^{n}_{t} \Vert_{L^2(\zeroT L^{2}(\Gampl))}^{2} }_{{II}}
+\rho\int_{\Gampl}\wtil^{n}_1\, \wtil^{n}_0\, dS
+\rho\sup_{t\in(0,T)}\int_{\Gampl}-\wtil^{n}_{t}\, \wtil^{n}\, dS\\
&+\frac{b}{2}\, \left\Vert \frac{1}{\sqrt{\alpha(0)}} \Delta p_0 \right\Vert^2
+ \frac{c}{2}\, \Vert \Dnu p_0 \Vert_{L^{2}(\Gamma_{a})}^{2}
\damppl{+\frac12\beta \Vert (-\Delpl)^{\gamma/2} \wtil_0 \Vert_{L^{2}(\Gampl)}^{2}}
\\
&\qquad+\frac{1}{b} \left\Vert \frac{1}{\sqrt{\alpha}} f \right\Vert_{L^1(\zeroT L^2(\Omega))}^{2}
+ \frac{1}{2\delta \rho^{2}} \Vert(-\Delpl)^{-1}\Vert_{L^2(\Omega)\to H^2(\Omega)}  \Vert|\tilde{h}\Vert|_{L^2(\zeroT \HminustwoGampl)}^2.
\end{aligned}
\end{equation}
In case $\Gamma_{D} \neq \emptyset $, we may estimate the underlined term $I$ by the following

\begin{equation}\label{estI}
\begin{aligned}
&\Vert \nabla p^n_{t} \Vert_{L^{2}(\zeroT L^{2}(\Omega)} 
\leq  C_\Delta( 
\Vert \Delta p^n_{t} \Vert_{L^{2}(\zeroT H^{-1}(\Omega))} +
\Vert \Dnu p^n_{t} \Vert_{L^{2}(\zeroT H^{-1/2}(\Gamma_{a}\cup\Gamma_{N}\cup\Gampl)} 
) \\
&\leq  C_\Delta( 
C^\Omega_{L^2,H^{-1}}\Vert \Delta p^n_{t} \Vert_{L^{2}(\zeroT L^{2}(\Omega)} 
+\Vert \Dnu p^n_{t} \Vert_{L^{2}(\zeroT H^{-1/2}(\Gamma_{a})} 
+\rho\Vert \wtil^n_{t} \Vert_{L^{2}(\zeroT H^{-1/2}(\Gamma_{a})} 
), 
\end{aligned}
\end{equation}
which together with \eqref{BCptil} allows us to bound I by the damping terms on the left hand side of \eqref{enest1}, that is, the third,  fourth, and seventh terms.
Estmiate \eqref{estI}
follows from elliptic estimates applied to the Laplace equation with mixed Neumann Dirichlet boundary conditions
\begin{align*}
{-}\Delta \phi &=f \text{ in }\Omega, \\
\Dnu \phi &=g \text{ on }\Gamma_{a} \cup \Gamma_{N} \cup \Gampl, \\
\phi&=0 \text{ on }\Gamma_{D}, 
\end{align*}
whose unique solution by the Lax-Milgram Theorem satisfies the inequality
\begin{equation}
\Vert \phi \Vert_{H^{1}} \leq {C_\Delta} (\Vert f \Vert_{H^{-1}}+ \Vert g \Vert_{H^{-1/2}(\Gamma_{N} \cup \Gampl \cup \Gamma_{a})})
\end{equation} 
see, e.g., \cite[p.~ 125]{Kesavan2003}.

Note that in case $\Gamma_{D} = \emptyset $, 
considering instead
\begin{align*}
-\Delta \phi +\phi&=f+\phi \text{ in }\Omega, \\
\Dnu \phi &=g \text{ on }\Gamma_{a} \cup \Gamma_{N} \cup \Gampl, 
\end{align*}
we only have 
\begin{equation}
\begin{aligned}
&\Vert \nabla p^n_{t} \Vert_{L^{2}(\zeroT L^{2}(\Omega){)}} \\
&\leq 
C_\Delta( 
C^\Omega_{L^2,H^{-1}}\Vert \Delta p^n_{t} \Vert_{L^{2}(\zeroT L^{2}(\Omega)} +
\Vert \Dnu p^n_{t} \Vert_{L^{2}(\zeroT H^{-1/2}(\Gamma_{a}\cup\Gamma_{N}\cup\Gampl)}
+\Vert p_{t} \Vert_{L^{2}(\zeroT H^{-1}(\Omega))})
\end{aligned}
\end{equation} 
which requires an additional estimate to control the $L^{2}$ norm of $p_{t}$. 
For this purpose we use the third term in \eqref{enest0}.

The term II in \eqref{eqpid} is bounded by the plate damping term with coefficient $\beta$  in \eqref{enest1}.

We recall that $\Delpl$ is the Laplace-Beltrami operator equipped with homogeneous Dirichlet boundary conditions, cf. \eqref{BCw}, which by elliptic regularity allows us to recover the full $H^2$ norm of $\wtil^{n}$. 

{
Combining \eqref{eqpid} with \eqref{enest1} {and the control of the $L^2$ norm by either Poincar\'e-Friedrichs or Step 2.b} yields the energy estimate
\begin{equation}\label{enest_eqp_n}
\begin{aligned}
&\Vert p^n_{t} \Vert_{L^\infty(0,t;H^1(\Omega))}^2
+(1+b)\Vert \Delta p^n \Vert_{L^\infty(0,t;L^2(\Omega))}^{2}
+ b\Vert \Delta p^n_{t} \Vert_{L^2(0,t;L^2(\Omega))}^{2}
+ \Vert \Delta p^n \Vert_{L^2(0,t;L^2(\Omega))}^{2}\\ 
&+ \Vert \Dnu p^n \Vert_{L^\infty(0,t;L^{2}(\Gamma_{a}))}^{2} 
+ \Vert \Dnu p^n_{t} \Vert_{L^2(0,t;L^{2}(\Gamma_{a}))}^{2}\\
&+ \Vert \wtil^n_{t} \Vert_{L^\infty(\zeroT L^{2}(\Gampl))}^{2}
+ \Vert \Delpl \wtil^n \Vert_{L^\infty(\zeroT L^{2}(\Gampl))}^{2}
+ {\Vert \Delpl \wtil^n \Vert_{L^2(\zeroT L^{2}(\Gampl))}^{2}}\\
&{}\damppl{+ \beta \Vert (-\Delpl)^{\gamma/2} \wtil^n_{t} \Vert_{L^{2}(0,t;L^{2}(\Gampl))}^{2}
+\beta \Vert (-\Delpl)^{\gamma/2} \wtil^n \Vert_{L^{\infty}(0,t;L^{2}(\Gampl))}^{2}}
\\
&\leq 
C \Bigl(\Vert p_1 \Vert_{H^1(\Omega)}^2 + (1+b)\Vert \Delta p_0 \Vert_{L^{2}(\Omega)}^2 
+ \Vert \Dnu p_0 \Vert_{L^{2}(\Gamma_{a})}^{2}
+ \Vert \wtil_1 \Vert_{L^{2}(\Gampl)}^{2} + \Vert \Delpl \wtil_0 \Vert_{L^{2}(\Gampl)}^{2}\\
&\qquad\qquad
\damppl{+\beta \Vert (-\Delpl)^{\gamma/2} \wtil_0 \Vert_{L^{2}(\Gampl)}^{2}}
+ F[f,\tilde{h}] 
+ \|\tilde{h}\|_{L^1(0,t;\HminustwoGampl)}^2 
\Bigr)
\end{aligned}
\end{equation}
whose weak limit can be written as \eqref{enest_eqp_E}.
}

\medskip

\noindent{\bf \underline{Step {3}} weak limits:}\\
According to Step 1, for each $n\in\mathbb{N}$, $\nabla p^{n}_{tt}$ and $\wtil^{n}_{tt}$ exist and lie in $L^2(\zeroT L^2(\Omega))$ and $L^2(\zeroT L^2(\Gampl))$, respectively (even though we cannot estimate these norms uniformly in $n$). Therefore, we can take the $L^2(0,T)$ inner product of \eqref{ODEp}, \eqref{ODEwtil} (that is, the Galerkin projection of \eqref{eqn:var_weak}) with arbitrary $C_c^\infty(0,T)$ functions $v$ and integrate by parts with respect to time to arrive at
\begin{equation}\label{var_weak_Galerkin}
\begin{aligned}
&\text{for all }(\phi^{n},\psi^{n})\in V^n, \ v\in C^\infty(0,T), \, v(T)=0\, : \\
&\int_0^T\Bigl\{\int_\Omega \Bigl(-\nabla p^{n}_{t}\cdot\nabla \phi^{n} v' +\tfrac{c^2}{\alpha}\Delta p^{n}\,\Delta\phi^{n}v+\tfrac{b}{\alpha}
\Delta p^{n}_t\,\Delta \phi^{n}v\Bigr)\, dx\\
&\qquad +\tfrac{\rho}{\kappa}\int_{\Gampl} \Bigl(-\rho \wtil^{n}_{t}\psi^{n} v'+\delta\Delpl\wtil^{n}\, \Delpl\psi^{n} v
\damppl{+\beta (-\Delpl)^\gamma\wtil^{n}_t\,\psi^{n} v}\Bigr) \,dS
+\int_{\Gamma_{a}} c\partial_\nu p^{n}_t\, \partial_\nu \psi^{n} v\, dS\Bigr\}\,dt\\
&= \int_0^T \Bigl\{\int_\Omega \tfrac{f}{\alpha}\,(-\Delta \phi^{n})\,dx +\tfrac{\rho}{\kappa}\int_{\Gampl} \tilde{h}\, \psi^{n}\, dS 
\Bigr\}\, v\,dt 
+ \Bigl(\int_\Omega \nabla p^{n}_1\cdot\nabla \phi^{n}\, dx +\int_{\Gampl} {\tfrac{ \rho^{2}}{\kappa}}  \Bigl(-\wtil^{n}_1\psi^{n}\,dS\Bigr)v(0)\,.
\end{aligned}
\end{equation}
The energy estimates from Step 2 allow us to extract a subsequence of the Galerkin approximations that weakly *  converges in $U$ (which is the dual of a separable space) to some $(p,\wtil)\in U$. Taking limits along this subsequence in \eqref{var_weak_Galerkin} 
relying on Lemma~\ref{lem:HDelta}
yields \eqref{eqn:var_weak} and by weak* lower semicontinuity of the energy functional on the left hand side of \eqref{enest}, the limit $(p,\wtil)$ also satisfies the energy estimate \eqref{enest}.

\medskip

\noindent{\bf \underline{Step {4}} attainment of the initial conditions:}\\
By uniform boundedness of 
$(p^{n},\wtil^{n})$ 
in $L^\infty(\zeroT {H_{\Delta}^2(\Omega)} \times H^2(\Gampl))\cap C([0,T],H^1(\Omega)\times L^2(\Gampl))
$, respectively, using \cite[Lemma 3.1.7]{zheng2004nonlinear} we can conclude that along the subsequence from Step 3 we have weak convergence of $(p^{n}(0),\wtil^{n}(0))$ in ${H_{\Delta}^2(\Omega)} \times H^2(\Gampl)$ to $(p(0),\wtil(0))$ and since at the same time $(p^{n}(0),\wtil^{n}(0))$ (the Galerkin projected initial data) converges to $(p_0,\wtil_0)$ even strongly in ${H_{\Delta}^2(\Omega)} \times H^2(\Gampl)$, the initial conditions are satisfied by $(p,\wtil)$ in ${H_{\Delta}^2(\Omega)} \times H^2(\Gampl)$.

\medskip

\noindent{{\bf \underline{Step {5}} higher regularity in time; {proof of (iii)}:}}\\
To prove (iii) we use the fact that by time differentiation, $(p',\tilde{w}'):=(p^{(n)}_t,\tilde{w}^n_t)$ solves \eqref{eqn:var} with $(f,\tilde{h})$ replaced by $(f',\tilde{h}')$, where $f'=f_t-\frac{\alpha_t}{\alpha}(f+c^2\Delta p^{n}+b\Delta p^{n}_t)$, $\tilde{h}'=\tilde{h}_t$.
Applying (i) with 
\[
\begin{aligned}
\Vert f'\Vert_{L^1(\zeroT L^2(\Omega))}
&\leq \Vert f_t\Vert_{L^1(\zeroT L^2(\Omega))}
+\Vert \tfrac{\alpha_t}{\alpha}\Vert_{L^1(\zeroT L^\infty(\Omega))}
\Vert f+c^2\Delta p^{n}+b\Delta p^{n}_t\Vert_{L^\infty(\zeroT L^2(\Omega))}\\
&\leq \Vert f_t\Vert_{L^1(\zeroT L^2(\Omega))}
+\bar{\alpha}\Bigl(\Vert f \Vert_{L^\infty(\zeroT L^2(\Omega))}
+ C \Bigl( \mathcal{E}[p,\wtil](0) + F[f,\tilde{h}]\Bigr)\Bigr)
\Bigr)
\end{aligned}
\]
and taking weak limits implies the assertion.

\medskip

\noindent{\bf \underline{Step {6}} uniqueness:}\\
In view of linearity it only remains to prove that the solution to the homogeneous equation vanishes.
To this end, note that for $(p,\wtil)\in U$ we can perform the testing from Step 2 of the proof of Proposition~\ref{prop:wellposed_lin_weak} to the infinite dimensional version of \eqref{eqn:var}. In the homogeneous case $f=0$, 
$\tilde{h}=0$, $(p_0,\wtil_0)=(0,0)$, $(p_1,\wtil_1)=(0,0)$, this results in $\mathcal{E}[p,\wtil]\equiv0$, thus $(p,\wtil)\equiv0$. 

\medskip

\noindent{{\bf \underline{Step 7}} proof of (iv):}\\
The proof of (iv) goes analogously to the one of (ii), 
{by testing the time differentiated equation (cf. Step 5) with $(p^n_t,\wtil^n_t)$.}
\end{proof}

\begin{Remark}[exponential decay]
In particular, if $f=0$, $\tilde{h}=0$, 
the energy estimate \eqref{enest_eqp_E} in Proposition~\ref{prop:wellposed_lin_weak} implies exponential decay of $\Vert \nabla p_{t} \Vert_{L^2(\Omega)}^2+\Vert \Delta p \Vert_{L^2(\Omega)}^2$ 
and --- via \eqref{BCptil} as well as  
$\Vert \Dnu{p} \Vert_{H^{-1/2}(\Gampl)} \leq C^{tr}_{1/2} \Bigl(\Vert \nabla p \Vert_{L^{2}(\Omega)} + \Vert \Delta p \Vert_{L^{2}(\Omega)}\Bigr)$, cf. Lemma~\ref{lem:div} --- 
also of $\Vert\wtil\Vert_{H^{-1/2}(\Gampl)}$.
\end{Remark}

To extend the result to inhomogeneous Dirichlet, Neumann, and impedance conditions, 
we make use of the decomposition $p=\bar{p}+\tilde{p}$, estimating $\bar{p}$ by Lemma~\ref{lem:extension} and $\tilde{p}$ by Proposition~\ref{prop:wellposed_lin_weak}. To obtain energy estimates on $p=\bar{p}+\tilde{p}$, we use the fact that $\sqrt{\Ep}$ is a seminorm, in particular it satisfies the triangle inequality. For the definitions of $\mathcal{E}$, $\Ep$, $F[f,\tilde{h}]$ and $G[a,g_D,g_N]$, we refer to \eqref{energy}, \eqref{F}, \eqref{G}.

\begin{Corollary}\label{cor:wellposed_lin}
Let, the regularity assumptions on the data of Lemma~\ref{lem:extension} and of Proposition~\ref{prop:wellposed_lin_weak} (i), (iii) 
as well as the compatibility conditions 
\begin{equation}\label{compat_cor}
\begin{aligned}
p_1+c\Dnu{p_0}=a(0), \ p_2+c\Dnu{p_1}=a_t(0)\ &\text{ on }\Gamma_a\\
p_0 =g_D(0), \ p_1 =g_{D\,t}(0), \ p_2 =g_{D\,tt}(0)\ &\text{ on }\Gamma_D\\
\Dnu{p_0} =g_N(0), \ \Dnu{p_1} =g_{N\,t}(0)\ \Dnu{p_2} =g_{N\,tt}(0)\ &\text{ on }\Gamma_N\\
\partial_\nu p_0 = -\rho\wtil_0, \  \partial_\nu p_1 = -\rho\wtil_1, \  \partial_\nu p_2 = -\rho\wtil_2\ &\text{ on }\Gampl
\end{aligned}
\end{equation}
where $p_2 := \tfrac{1}{\alpha(0)}\Bigl(c^2\Delta p_{0}+b\Delta p_1+f(0)\Bigr)$,
$\tilde{w}_2\:=\frac{1}{\rho}\Bigl(-\delta\Delpl^2\tilde{w}_0\damppl{-\beta(-\Delpl)^\gamma\tilde{w}_1}+\kappa p_2+\tilde{h}(0)\Bigr)$

be satisfied.

Then there exists a solution $(p,\wtil)$ to \eqref{pwlin} in $U'$ and the estimate 
\begin{equation}\label{enest_prime_inhom}
\begin{aligned}
\Vert\mathcal{E}_1[p,\wtil]\Vert_{L^\infty(0,T)}
\leq C \Bigl( \mathcal{E}_1[p,\wtil](0) + F[f,\tilde{h}] + F[f_t,\tilde{h}_t]
+ G[a,g_D,g_N]+ G[a_t,g_{D\,t},g_{N\,t}]\Bigr)
\end{aligned}
\end{equation}
holds.
The solution is unique in $U'$.

If additionally the assumptions of Proposition~\ref{prop:wellposed_lin_weak} (ii), (iv) hold, then 
\begin{equation}\label{enest_prime_inhom_eqp}
\begin{aligned}
&\Vert\mathcal{E}_1[p,\wtil]\Vert_{L^\infty(0,T)}
+\Vert\mathcal{E}_1[p,\wtil]\Vert_{L^1(0,T)} \\
&\leq 
C \Bigl(\mathcal{E}_1[p,\wtil](0)
+ b \Vert \Delta p_0 \Vert_{L^2(\Omega)}^{2}
+ b \Vert \Delta p_1 \Vert_{L^2(\Omega)}^{2}
+\Vert\Dnu{p_0}\Vert_{L^2(\Gamma_a)}^2
+\Vert\Dnu{p_1}\Vert_{L^2(\Gamma_a)}^2
\\
&\qquad\qquad\damppl{+\beta \Vert (-\Delpl)^{\gamma/2} \wtil_0 \Vert_{L^{2}(\Gampl)}^{2}}
\damppl{+\beta \Vert (-\Delpl)^{\gamma/2} \wtil_1 \Vert_{L^{2}(\Gampl)}^{2}}
\\
&\qquad\qquad+ F[f,\tilde{h}]  + F[f_t,\tilde{h}_t]
+ G[a,g_D,g_N]+ G[a_t,g_{D\,t},g_{N\,t}]
\Bigr)
\end{aligned}
\end{equation}
The constant $C$ depends on the coefficients $b$, $c$, $\rho$, $\delta$, $\kappa$, 
$\Vert\frac{1}{\alpha}\Vert_{L^\infty(\zeroT L^\infty(\Omega))}$,  
$\Vert\alpha_t\Vert_{L^1(\zeroT L^\infty(\Omega))}$, but not directly on $T$.
\end{Corollary}

\section{{Local and Global in Time Well-posedness of} the Nonlinear Model} \label{sec:nl}

To {first of all} prove local in time well-posedness of the quasilinear coupled system
\begin{equation}\label{pw}
\begin{aligned}
(1-2k p)\, p_{tt}  - c^{2}\Delta p - b \Delta \tilde{p}_{t}  &= 2k (p_{t})^{2} ~~\mbox{in}~~ \Omega \times [0,T] \\
 c\Dnu{ p}+p_{t} &= a  ~~\mbox{in}~~ \Gamma_{a} \times [0,T]  \\
 p&= g_D  ~~\mbox{on}~~ \Gamma_{D} \times [0,T] \\
\Dnu{ p}&=g_N ~~\mbox{on}~~ \Gamma_{N} \times [0,T] \\
\Dnu{ p} &= - \rho \wtil  ~~\mbox{on}~~ \Gampl \times [0,T] \\
\rho\wtil_{tt} + \delta\Delpl^{2} \wtil \damppl{+\beta(-\Delpl)^{\gamma} \wtil_{t}} &= \kappa p_{tt} +\tilde{h}  ~~\mbox{on}~~ \Gampl \times [0,T] \\
\tilde{p}(0)&= p_{0}, ~~\tilde{p}_{t}(0)= p_{1}.
\end{aligned}
\end{equation}
we construct $(p,\wtil)$ as a fixed point of the map 
\[
\Lambda: W\mapsto U', \quad (q,v) \to (p, \wtil) \text{ solving \eqref{pwlin} with $\alpha =1-2k q$, $f=2k (q_{t})^{2}$} 
\]
on the following subset $W$ of the solution space $U'$ (cf. \eqref{Uprime}, \eqref{U})
\begin{align*}
W = \{ ( q, v) \in U': & \Vert q \Vert_{L^\infty(\zeroT L^\infty(\Omega))} \leq M_{1},\
  	{\Vert\mathcal{E}[q,v]\Vert_{L^\infty(0,T)}} \leq M_{2}, \ 
	{\Vert\mathcal{E}[q_t,v_t]\Vert_{L^\infty(0,T)}}\leq M_{3},\\
 & (q(0),q_{t}(0))= (p_{0}, p_{1}), \
  (v(0),v_{t}(0))= (\wtil_{0}, \wtil_{1})\\
& c\Dnu{q}+q=a\text{ on }\Gamma_a, \ 
	q=g_D \text{ on }\Gamma_D, \ 
	\Dnu{q}=g_N \text{ on }\Gamma_N, \ 
	\Dnu{q}=-\rho v \text{ on }\Gampl 
 \},
\end{align*}
where $0<M_{1} <\frac{1}{2k}$ and $0<M_{2}$, $0<M_{3}$  are to be chosen in a suitable way, where the first of these requirements is supposed to avoid degeneracy of the coefficient $\alpha =1-2k q$.

For the latter purpose, a crucial step in the proof invariance of the set $W$ under the map $\Lambda$ is to control the $L^\infty([0,T]\times \Omega)$ norm of $q$ in terms of its energy.
This is here impeded by the fact that the reduced boundary regularity limits applicability of elliptic regularity in the sense that we cannot conclude $H^2(\Omega)\subseteq H^2_\Delta(\Omega)$ (and then use continuity of the embedding $H^2(\Omega)\to L^\infty(\Omega)$).
We thus proceed with some results allowing us to obtain an $L^\infty(\Omega)$ bound on $p(t)$ from reasonable assumptions on the boundary data.

{
To this end, in case of nonempty Dirichlet boundary, we make use of the following result, 
\begin{Lemma}\label{Plum}(\cite[Theorem 5]{Plum1992})
Suppose that $\Omega$, $\Gamma_D$ and $\partial\Omega\setminus\Gamma_D$ are regular in the sense that 
for any $r$ in a dense subset of $L^2(\Omega)$, the elliptic problem $-\Delta u + u= r$ has an $H^2(\Omega)$ solution such that $u\vert_{\Gamma_D}$, $\Dnu{u}\vert_{\Gamma_D}$ are essentially bounded
and $z\in H^2(\Omega)$ exist such that
\[
\text{essinf}_{x\in\Omega} z(x) >0, \quad
\text{esssup}_{x\in\Omega} \Delta z(x) <\infty, \quad
z\geq1 \text{ a.e. on }\Gamma_D, \quad
\Dnu{z}\geq1 \text{ a.e. on }\partial\Omega\setminus\Gamma_D
\]
Then there exist $K>0$ such that any solution $u$ to 
\begin{equation}\label{mixedbc}
\begin{aligned}
-\Delta u + u&= \ff \textup{ in }\Omega\\
u &= \gg_D \textup{ on }\Gamma_D\\
\Dnu{u} &= \gg_N \textup{ on }\partial\Omega\setminus\Gamma_D
\end{aligned}
\end{equation}
satisfies the estimate
\[
\Vert u\Vert_{L^\infty(\Omega)}\leq K (\Vert \ff\Vert_{L^2(\Omega)} +\Vert \gg_D\Vert_{L^\infty(\Gamma_D)}+\Vert \gg_N\Vert_{L^\infty(\Gamma_D)})
\]
\end{Lemma}
We mention in passing that \cite[Theorem 5]{Plum1992} also gives an explicit expression of $K$ in terms of $z$, as well as examples of domains satisfying these conditions referring also to \cite[Chapter 3, Section 9] {LadyzhenskayaUraltseva}.
We also point to \cite{Dauge1992,DekkersRozanovaTeplyaev,Savare1997}.
}

\begin{Lemma}\label{Linfty}
Let $s\in(0,\frac12)$, $\ff\in L^2(\Omega)$, {$\gg_D\in L^\infty(\Gamma_D)$,} ${\gg_N}\in X_{\gg}$, and either
\begin{itemize}
\item[(a)] $\Gamma_D\not=\emptyset$ and the regularity conditions of Lemma~\ref{Plum} on $\partial\Omega$, $\Gamma_D$ hold and $X_{\gg}= H^s(\partial\Omega\setminus\Gamma_D)\cap L^\infty(\partial\Omega\setminus\Gamma_D)$ or
\item[(b)] $\Gamma_D=\emptyset$, $X_{\gg}= H^s(\partial\Omega)$. 
\end{itemize}
Then there exist a constant $K$  independent of ${\gg_D}$, ${\gg_N}$, $\ff$, such that for any solution $u$ to \eqref{mixedbc}, the estimate
\[
\Vert u\Vert_{L^\infty(\Omega)}\leq K (\Vert \ff\Vert_{L^2(\Omega)} 
{
+\Vert \gg_D\Vert_{L^\infty(\Gamma_D)}+\Vert \gg_N\Vert_{X_{\gg}}
})
\]
holds. In case (b) we additionally have 
\[
\Vert u\Vert_{H^{3/2+s}(\Omega)}\leq K (\Vert \ff\Vert_{L^2(\Omega)} 
+\Vert \gg_N\Vert_{X_{\gg}}
).
\]
\end{Lemma}
\begin{proof}
In case (a) we apply Lemma~\ref{Plum}.
In case (b) the assertion follows from standard results for the Laplace Neumann problem together with continuity of the embedding $H^{3/2+s}(\Omega)\to L^\infty(\Omega)$
\end{proof}

We use Lemma~\ref{Linfty} in order to bound $\|\alpha\|_{L^\infty}$, $\|\frac{1}{\alpha}\|_{L^\infty}$, $\|\alpha_t\|_{L^1(\zeroT L^\infty(\Omega))}$ and in case of $\Gamma_D=\emptyset$ also 
$\|\nabla\alpha\|_{L^2(\zeroT L^3(\Omega))}$, 
$\|\nabla\alpha\|_{L^\infty(\zeroT L^3(\Omega))}$, 
as required by Proposition~\ref{prop:wellposed_lin_weak}.

\begin{Lemma}
\label{alpha}
Let $M_{1}\leq \frac{1}{4k}$ and $M_{2}(1+T^{1/2})$, $M_{3}(1+T)$ as well as 
$\Vert g_D\Vert_{L^\infty(\zeroT H^{1+s}(\Gamma_D))}$, 
$\Vert g_N\Vert_{L^\infty(\zeroT H^s(\Gamma_N))}$, 
{$\Vert g_{N\,t}\Vert_{L^1(\zeroT H^s(\Gamma_N))}$} 
be small enough for some $s\in(0,1/2)$, where 
\\
in case (a) $\Gamma_D\not=\emptyset$ we assume that $\Gamma_a=\emptyset$ and the regularity conditions from  Lemma~\ref{Plum} on $\partial\Omega$, $\Gamma_D$ on $\partial\Omega$, $\Gamma_D$ are satisfied 

Then for any $(q,v) \in W$,
the functions $\alpha = 1-2kq$, $f = 2kq_t^2$ satisfy the following.
\begin{enumerate}
\item[1.]
For all $(t,x) \in \Omega \times [0,T]$, we have $\frac{1}{2}  \leq  \alpha(t,x) \leq \frac{3}{2}$.
\item[2.]
$ \alpha_{t}  \in L^{1}(\zeroT  L^{\infty}( \Omega))$ with 
$\bar{\alpha} = \Vert\frac{\alpha_t}{\alpha}\Vert_{L^1(\zeroT L^\infty(\Omega))}<1/2$
\item[3.]
In case $\Gamma_D=\emptyset$, $\nabla\alpha  \in L^{2}(\zeroT  L^{3}( \Omega))$ with 
{
$\tilde{\alpha}_2 = \Vert\alpha^{-3/2}\nabla \alpha \Vert_{L^2(\zeroT L^3(\Omega))}<1/4$,\\
$\tilde{\alpha}_\infty = \Vert\alpha^{-3/2}\nabla \alpha \Vert_{L^\infty(\zeroT L^3(\Omega))}<1/4$
}
\item[4.] 
$f\in W^{1,1}(\zeroT L^2(\Omega))$.
\end{enumerate}
\end{Lemma}

\begin{proof}

The estimate 1. immediately follows from the definition of $W$ with $M_{1}\leq \frac{1}{4k}$ and 
\[
1-\Vert\alpha-1\Vert_{L^\infty(\zeroT L^\infty(\Omega))}\leq \alpha(t,x)\leq 
1+\Vert\alpha-1\Vert_{L^\infty(\zeroT L^\infty(\Omega))}
\]
with 
\[
\Vert\alpha-1\Vert_{L^\infty(\zeroT L^\infty(\Omega))}
= 2k\Vert q\Vert_{L^\infty(\zeroT L^\infty(\Omega))} \leq 2k M_{1}\leq\frac12
\]

\medskip

To prove 2. we invoke the two cases of Lemma~\ref{Linfty} separately:

In case (b) $\Gamma_D=\emptyset$, relying on the control of the $L^2(\Omega)$ norm of $q_t$ and $\Delta q_t$ as established in the proof of Proposition~\ref{prop:wellposed_lin_weak}), we can simply set
\begin{equation}\label{frakgfrakf}
{\gg_N}=\begin{cases} 
-\tfrac{1}{c} q_{tt}(t) &\textup{ on }\Gamma_a\\
-\rho v_t(t) &\textup{ on }\Gampl\\
g_{N\, t}&\textup{ on }\Gamma_N
\end{cases}
\qquad \ff = -\Delta q_t(t)+q_t(t).
\end{equation}
to obtain
\begin{equation}\label{alphat}
\begin{aligned}
&\Vert\alpha_t\Vert_{L^1(\zeroT L^\infty(\Omega))}
=2k\Vert q_t\Vert_{L^1(\zeroT L^\infty(\Omega))}\\
&\leq 2k K\Bigl(\Vert -\Delta q_t+q_t\Vert_{L^1(\zeroT L^2(\Omega))} 
+C^{tr}_{H^s}\tfrac{1}{c}\Vert q_{tt}\Vert_{L^1(\zeroT H^{s+1/2}(\Omega))}
+\rho\Vert v_{t}\Vert_{L^1(\zeroT H^s(\Gampl))}
+\Vert g_{N\,t}\Vert_{L^1(\zeroT H^s(\Gamma_N))}
\Bigr)\\
&\leq 2kK (\tilde{C} \sqrt{\sup_{t\in(0,T)}\mathcal{E}[q_t,v_t]+\int_0^T\mathcal{E}[q_t,v_t]\, ds} +  \Vert g_{N\,t}\Vert_{L^1(\zeroT H^s(\Gamma_N))})\\
&\leq 2kK (\tilde{C} \sqrt{(1+T)\sup_{t\in(0,T)}\mathcal{E}[q_t,v_t]} + \Vert g_{N\,t}\Vert_{L^1(\zeroT H^s(\Gamma_N))}).
\end{aligned}
\end{equation}
Here we have used the fact that for $s\in[0,1/2)$ the concatenation of $H^s$ boundary functions is again an $H^s$ boundary function, see, e.g., 
\cite[Corollary 1.4.4.5.]{Grisvard} combined with the arguments in \cite[Appendix]{BK-Peichl}.
Thus by choosing 
{$s\geq0$}
$(1+T)M_3$ and $\Vert g_{N\,t}\Vert_{L^1(\zeroT H^s(\Gamma_N))}$ small enough, so that $2kK (\tilde{C} M_3 + \Vert g_{N\,t}\Vert_{L^1(\zeroT H^s(\Gamma_N))})\leq 1/2$, we obtain 2.

Likewise 
{
\[
\begin{aligned}
\Vert\nabla\alpha\Vert_{L^\infty(\zeroT L^3(\Omega))}
&\leq C^\Omega_{H^{s+1/2},L^3}\Vert\alpha\Vert_{L^\infty(\zeroT H^{s+3/2}(\Omega))}
= 2k  C^\Omega_{H^{s+1/2},L^3}\Vert q\Vert_{L^\infty(\zeroT H^{s+3/2}(\Omega))}\\
&\leq 2kK  C^\Omega_{H^{s+1/2},L^3} (\tilde{C} (1+T)\sup_{t\in(0,T)}\mathcal{E}[q,v](t) + \Vert g_N\Vert_{L^\infty(\zeroT H^s(\Gamma_N)}).
\end{aligned}
\]
and similarly for $\Vert\nabla\alpha\Vert_{L^2(\zeroT L^3(\Omega))}$ in $\tilde{\alpha}_2$}
thus implying 3. provided $(1+T)M_2$, $(1+T)M_3$, $\Vert g_N\Vert_{L^\infty(\zeroT H^s(\Gamma_N))}$ are small enough.

\medskip

In case (a) $\Gamma_D\not=\emptyset$, the energy controls $\Vert q_t(t)\|_{L^2(\Omega)}$ due to Poincar\'e's inequality and we set ${\gg_N}$, $\ff$ according to \eqref{frakgfrakf}.
From $v_t(t)\in H^r(\Gampl)\subseteq L^\infty(\Gampl)$ and $q_{tt}\in H^r(\Gamma_a)\subseteq L^\infty(\Gamma_a)$ with $r>1$ we would obtain ${\gg_N}\in H^s(\Gamma_N)\cap L^\infty(\Gamma_N)$ for any $s\in(0,\frac12)$ and could use  Lemma~\ref{Linfty} (a) to conclude 2.

However, the available energy estimates do not provide $q_{tt}(t)\in H^r(\Gamma_a)$ for any $r>1$.
To achieve this, we would have to derive higher order energy estimates, which would lead to stronger assumptions on smoothness of the data.
In order to avoid this problem we assume $\Gamma_a=\emptyset$ in case (a) $\Gamma_D\not=\emptyset$.  
Therewith it suffices to bound 
\[
\Vert v_t(t)\Vert_{H^r(\Gampl)}\leq C^{\Gampl}_{H^2,H^r}\mathcal{E}[q_t,v_t](t), 
\]
and the assertion follows analogously to case (b).

\medskip

For $f$ and $f_t$ we obtain 
\[
\Vert f\Vert_{L^1(\zeroT L^2(\Omega))}
= 2k \Vert q_t\Vert_{L^1(\zeroT L^4(\Omega))}^2
\leq 2k (C^{\Omega}_{H^1,H^4})^2 \Vert q_t\Vert_{L^1(\zeroT H^1(\Omega))}^2,
\]
\[
\Vert f_t\Vert_{L^1(\zeroT L^2(\Omega))}
= 4k \Vert q_{tt} q_t\Vert_{L^1(\zeroT L^2(\Omega))}
\leq 4k \Vert q_{tt}\Vert_{L^\infty(\zeroT L^2(\Omega))} \Vert q_t\Vert_{L^1(\zeroT L^\infty(\Omega))}
\]
where {$\Vert q_t\Vert_{L^1(\zeroT H^1(\Omega))}$} and $\Vert q_{tt}\Vert_{L^\infty(\zeroT L^2(\Omega))}$
can be bounded by some constant multiple of 
{$\int_0^T\mathcal{E}[q,v](t)\, dt$} and $\sup_{t\in(0,T)}\mathcal{E}[q_t,v_t](t)$, respectively, and 
$\Vert q_t\Vert_{L^1(\zeroT L^\infty(\Omega))}$ has already been bounded in the proof of 2.
\end{proof}

\begin{Theorem}\label{theo:wellposed}
In case $\Gamma_D\not=\emptyset$ assume that $\Gamma_a=\emptyset$ and the regularity conditions from  Lemma~\ref{Plum} on $\partial\Omega$, $\Gamma_D$ are satisfied.

Let $T>0$ be arbitrary, then there exists constants  $m_{1} >0$, $m_{2} >0$ and $m_{3} >0$ possibly depending on $T$ 
such that 
for  initial conditions $(p_{0}, p_{1}, \wtil_{0}, \wtil_{1})$ and data {$(a,g_D,g_N,\tilde{h})$} satisfying the regularity assumptions of Corollary~\ref{cor:wellposed_lin} 
as well as the compatibility conditions \eqref{compat_cor}
with $p_2 := \tfrac{1}{1-2kp_0}\Bigl(c^2\Delta p_{0}+b\Delta p_1+2kp_1^2\Bigr)$,
$\tilde{w}_2\:=\frac{1}{\rho}\Bigl(-\delta\Delpl^2\tilde{w}_0\damppl{-\beta(-\Delpl)^\gamma\tilde{w}_1}+\kappa p_2+\tilde{h}(0)\Bigr)$
and 
\[
\begin{aligned}
&\mathcal{E}_1[p,\wtil](0) \leq m_{1},\\
&F[{2k\bar{p}_t^2,\tilde{h}+\bar{p}_{tt}}]
+F[{4k\bar{p}_t\bar{p}_{tt},\tilde{h}_t+\bar{p}_{ttt}}] + G[a,g_D,g_N]  {+ G[a,g_{D\, t},g_{N\, t}]} \leq m_{2}, \\ 
&\Vert g_D\Vert_{L^\infty(\zeroT H^{1+s}(\Gamma_D))} 
+ \Vert g_N\Vert_{L^\infty(\zeroT H^s(\Gamma_N))}
{+\Vert g_{N\,t}\Vert_{L^1(\zeroT H^s(\Gamma_N))}} 
\leq m_{3}
\end{aligned}
\]
for some $s\in(0,1/2)$,
there exists a unique solution $(p, \wtil)$ in the set $W$ to 
\eqref{pw}.

This solution satisfies the energy estimate
\begin{equation}\label{enest_eqp_E1_nl}
\begin{aligned}
&\mathcal{E}_1[p,\wtil](t)+\int_0^t \mathcal{E}_1[p,\wtil](s)\, ds\leq 
C \Bigl(\int_0^t \mathcal{E}_1[p,\wtil](s)\, ds\Bigr)^2 \\
&+C \Bigl(\mathcal{E}_1[p,\wtil](0)
+ b \Vert \Delta p_0 \Vert_{L^2(\Omega)}^{2}
+ b \Vert \Delta p_1 \Vert_{L^2(\Omega)}^{2}
+\Vert\Dnu{p_0}\Vert_{L^2(\Gamma_a)}^2
+\Vert\Dnu{p_1}\Vert_{L^2(\Gamma_a)}^2
\\
&\qquad\qquad\damppl{+\beta \Vert (-\Delpl)^{\gamma/2} \wtil_0 \Vert_{L^{2}(\Gampl)}^{2}}
\damppl{+\beta \Vert (-\Delpl)^{\gamma/2} \wtil_1 \Vert_{L^{2}(\Gampl)}^{2}}
+ F[0,\tilde{h}]  + F[0,\tilde{h}_t]
\Bigr)
\end{aligned}
\end{equation}

If additionally 
$\tilde{h}+\bar{p}_{tt}\in W^{1,1}(0,\infty; \HminustwoGampl)$, 
$g_{N\,t}\in W^{2,1}(0,{\infty};H^{-1/2}(\Gamma_N))$,
then $m_1,m_2,m_3$ can be chosen independently of $T$ and the solution exists globally in time, that is, the assertion above is valid with $T=\infty$.

{If additionally the forcing data $a,g_D,g_N,\tilde{h}$ vanish, then $\mathcal{E}_1[p,\wtil](t)$ decays exponentially
}
\end{Theorem}
\begin{proof}
\noindent{\bf \underline{Step 1} Invariance of solution map:}\\
We let $(q,v) \in W$ and consider equation \eqref{pwlin}
with
$ \alpha = 1- 2kq$
and 
$ f = 2k (q_{t})^{2}$.
By Lemma \ref{alpha}, $\alpha$ and $f$ satisfy the assumptions of Proposition \ref{prop:wellposed_lin_weak} (i), (iii) and thus, due to our assumptions on the data there exists 
a solution $\Lambda(q,v)=(p, \wtil)\in U'$. 
To show that $\Lambda$ maps $W$ into itself for suitably chosen $m_{1}, m_{2} >0$, we employ the energy estimate \eqref{enest_prime_inhom} in Corollary ~\ref{cor:wellposed_lin}, that together with our assumptions on smallness of the data implies 
\[
\mathcal{E}[p,\wtil](t)+\mathcal{E}[p_t,\wtil_t](t) 
\leq C \Bigl(m_1+m_2).
\]
Choosing $m_1$, $m_2$ small enough thus allows us to conclude $\mathcal{E}[p,\wtil](t)\leq M_2$, $\mathcal{E}[p_t,\wtil_t](t)\leq M_3$
Additionally, analogously to Lemma~\ref{alpha} 2., we estimate the $L^\infty(\zeroT L^\infty(\Omega))$ norm of $p$ using Lemma ~\ref{Linfty} to obtain 
\[
\Vert p \Vert_{L^\infty(\zeroT L^\infty(\Omega))}\leq C( m_{1}+m_{2}+m_{3} )
\] 
{thus, by choosing $m_{1}, m_{2}, m_{3}$ small enough, $\Vert p \Vert_{L^\infty(\zeroT L^\infty(\Omega))}\leq M_{1}$.}

\noindent{\bf \underline{Step 2} Contraction estimates:}\\ 
We next show that the solution map $\Lambda: (q,v) \to (p,\wtil)$ is a contraction on $W$ in the topology of the space $U$.
In particular, let $(p_1,\wtil_1)$ and $(p_2, \wtil_2)$ be the solutions {of \eqref{pwlin}}corresponding to $(q_1,v_1)$ and $(q_2,v_2) $ in $W$. We denote the differences
of solutions by $(\hat{p}, \hat{w})$ and $(\hat{q},\hat{v})$; that is $\hat{p}=p_1-p_2$, $\hat{w} = \wtil_1-\wtil_2$, $\hat{q}=q_1-q_2$, $\hat{v}=v_1-v_2$. We also denote by $\alpha_1 = 1-2kq_1$, $\alpha_2 = 1-2kq_2$ while  $f_{1}= 2kq_{1\,t}^2$ and $f_2=2kq_{2\,t}^2$.
Then $(\hat{p}, \hat{w})$ satisfies \eqref{pwtil} with $(f,\tilde{h})$ replaced by $(\hat{f},0)$ where 
$\hat{f}=f_1-f_2-\frac{\alpha_1-\alpha_2}{\alpha_1\alpha_2}(c^2\Delta p_2+b\Delta p_{2\,t})$ and homogeneous initial conditions.
We apply Proposition~\ref{prop:wellposed_lin_weak} (i) with the estimate
\[
\begin{aligned}
\Vert\hat{f}\Vert_{L^1(\zeroT L^2(\Omega))}
\leq &2k \Vert q_{1\,t}+q_{2\,t}\Vert_{L^2(\zeroT L^4(\Omega))}\Vert \hat{q}_t\Vert_{L^2(\zeroT L^4(\Omega))}
\\
&+8k \Vert \hat{q}\Vert_{L^2(\zeroT L^\infty(\Omega))}(c^2{
\Vert\Delta p_2\Vert_{L^2(\zeroT L^2(\Omega))}}+b\Vert\Delta p_{2\,t}\Vert_{L^2(\zeroT L^2(\Omega))}),
\end{aligned}
\]
where we have used 
\[
f_1-f_2=2k (q_{1\,t}+q_{2\,t})(q_{1\,t}-q_{2\,t}),\qquad
\left|\frac{\alpha_1-\alpha_2}{\alpha_1\alpha_2}\right| \leq 4 |\alpha_1-\alpha_2| = 8k|q_1-q_2|,
\] 
Together with continuity of the embedding $H^1(\Omega)\to L^4(\Omega)$ and Lemma~\ref{Linfty} this yields
{
\begin{equation}\label{fhat}
\begin{aligned}
\Vert (\hat{p},\hat{w})\Vert_U 
=& \sqrt{\Vert\mathcal{E}[\hat{p},\hat{w}]\Vert_{L^\infty(0,T)}}
\leq \sqrt{C F[\hat{f},0]} = \sqrt{C} \Vert\hat{f}\Vert_{L^1(\zeroT L^2(\Omega))}\\
\leq& 2k  \sqrt{C} \Vert (\hat{q},0)\Vert_U 
\sqrt{ 
\mathcal{E}[q_1,0]+\mathcal{E}[q_2,0]
+\Vert\mathcal{E}[p_2,0]\Vert_{L^\infty(0,T)}+\Vert\mathcal{E}[p_2,0]\Vert_{L^2(0,T)}
}\\
\leq& 2k \sqrt{C} \Vert (\hat{q},\hat{v})\Vert_U (1+\sqrt{T})\sqrt{3M_2}.
\end{aligned}
\end{equation}
We then choose $M_{2}$ sufficiently small so that 
\[
2k \sqrt{C} (1+\sqrt{T})\sqrt{3M_2} < 1
\]
and we get a contraction estimate.
} 

\medskip

\noindent
{\bf Local in time well-posedness:}\\ 
The Picard sequence $(p_n,\wtil_n)$ therefore converges to a fixed point of $\Lambda$ (that is, a solution $(p,\wtil)$ to \eqref{pw}) in the norm topology of the weaker space $U$. 
Moreover, since $W$ is a ball in $U'$ and thus weakly-* compact in $U'$ due to the Banach-Alaoglu Theorem, we can  extract a weakly* convergent subsequence $(p_{n_k},\wtil_{n_k})$ with limit $(\bar{p},\bar{\wtil})\in W$. Due to uniqueness of limits we have $(p,\wtil)=(\bar{p},\bar{\wtil})\in W$.

The energy estimate \eqref{enest_eqp_E1_nl} follows from \eqref{enest_eqp_E1_nl} together with 
$F(f,\tilde{h}]=F[f,0]+F[0,\tilde{h}]$ cf. \eqref{F} and the estimate
\[
\begin{aligned}
F[f,0]  + F[f_t,0]
= \Vert f \Vert_{L^1(0,t;L^2(\Omega))}^{2} +\Vert f_t \Vert_{L^1(0,t;L^2(\Omega))}^{2}
& =4k^2 \Bigl(\Vert p_t\Vert_{L^2(0,t;L^4(\Omega))}^4+2\Vert p_t\, p_{tt}\Vert_{L^1(0,t;L^2(\Omega))}^2\Bigr)  \\
& \leq 12k^2 (C^\Omega_{H^1,L^4})^2 \Bigl(\int_0^t \mathcal{E}_1(s)\, ds\Bigr)^2
\end{aligned}
\]

\medskip

\noindent
{\bf Global in time well-posedness:}\\ 
If additionally 
$\tilde{h}\in W^{1,1}(0,\infty; \HminustwoGampl)$, 
$g_{N\,t}\in W^{2,1}(0,\infty;H^{-1/2}(\Gamma_N))$, 
we can replace the definition of $W$ by 
\begin{align*}
W = \{ ( q, v) \in U': & \Vert q \Vert_{L^\infty(\zeroT L^\infty(\Omega))} \leq M_{1},\
  {\Vert\mathcal{E}[q,v]\Vert_{L^\infty(0,T)}+\Vert\mathcal{E}[q,v]\Vert_{L^1(0,T)}} \leq M_{2}, \\ 
  &{\Vert\mathcal{E}[q_t,v_t]\Vert_{L^\infty(0,T)}+\Vert\mathcal{E}[q_t,v_t]\Vert_{L^1(0,T)}}\leq M_{3},\\
 & (q(0),q_{t}(0))= (p_{0}, p_{1}), \
  (v(0),v_{t}(0))= (\wtil_{0}, \wtil_{1})\\
& c\Dnu{q}+q=a\text{ on }\Gamma_a, \ 
	q=g_D \text{ on }\Gamma_D, \ 
	\Dnu{q}=g_N \text{ on }\Gamma_D, \ 
	\Dnu{q}=-\rho v \text{ on }\Gampl 
 \},
\end{align*}
and use the fact that Lemma~\ref{alpha} then holds with
$M_{2}$, $M_{3}$ as well as 
$\Vert g_D\Vert_{L^\infty(\zeroT H^{1+s}(\Gamma_D))}$, $\Vert g_N\Vert_{L^\infty(\zeroT H^s(\Gamma_N))}$, 
$\Vert g_{N\,t}\Vert_{L^1(\zeroT H^s(\Gamma_N))}$ 
small enough, cf., e.g., the next to last estimates in \eqref{alphat}, \eqref{fhat}.

Exponential decay follows from the energy estimate \eqref{enest_eqp_E1_nl} by barrier's method, cf., e.g., \cite{KL09Westervelt}.
\end{proof}

\begin{Remark}
\label{rem:regularity}
When dropping the possibility of having Dirichlet and absorbing boundary conditions, that is, consider the case $\partial\Omega=\Gamma_N\cup\Gampl$, $\partial\Gamma_N=\partial\Gampl$, we can actually achieve more spatial regularity.
Indeed, replacing part of the hinged boundary condition by the trace of the Neumann data  
\[
\rho\wtil=-g_N\vert_{\partial\Gamma_N} \, \quad \Delta \wtil =0 \quad \text{ on }\partial\Gampl,
\]
and assuming $g_N(t)\in H^{1/2}(\Gamma_N)$,
via \eqref{BCptil} we achieve $\Dnu{p}\in H^{1/2}(\partial\Omega)$ and thus, together with $p(t)\in H^2_\Delta(\Omega)$, recover $p(t)$ in $H^2(\Omega)$ by means of elliptic regularity. The same holds true for the first and second time derivative of $p$ so that we end up with the solution spaces
\[
\begin{aligned}
U=&(W^{1,\infty}(\zeroT H^1(\Omega))\cap H^1(\zeroT H^2(\Omega)))\times (W^{1,\infty}(\zeroT L^2(\Gampl))\cap L^\infty(\zeroT H^2(\Gampl)))\\
U'=&\{(p,\wtil)\in U\, : \, (p_t,\wtil_t)\in U\}
\end{aligned}
\]
in place of \eqref{U}, \eqref{Uprime}. 
\end{Remark}

{
\section{Outlook}\label{sec:outlook}
We have kept track of the damping coefficients $b$ and $\beta$ in the estimates in order to open up the possibility of considering vanishing damping (see, e.g., \cite{btozero} for Westervelt alone). While this appears hopeless for $b$, it might be possible to recover some of estimates in a $\beta$ uniform way.  
A similar question arises for boundary damping if we replace \eqref{BCa} by $p_{t} +\sigma\Dnu{p}=a$ considering $\sigma\to0$.
\\
As already mentioned above, $g_N$, and $h$ could be used as controls and an analysis of related optimization problems along the lines of, e.g., \cite{ClasonBK:2009} could be of interest.
\\
Likewise, optimizing the shape of $\Gampl$ could be of practical relevance. 
}
\section*{Appendix}
\begin{proof}(Lemma~\ref{lem:extension})\\  
Since the Galerkin discretiztion step can be done very similarly to, e.g., 
\cite{Nikolic:2015}, we focus on energy estimates here. 

Differentiating \eqref{pbar} twice with respect to time and testing with $-\frac{1}{\alpha}\Delta \bar{p}_{ttt}$, we obtain 
\[
\begin{aligned}
0=& \frac12 \frac{d}{dt} \Vert \nabla \bar{p}_{ttt} \Vert^{2} 
-\int_{\partial\Omega} \bar{p}_{tttt}\Dnu{\bar{p}_{ttt}}\, dS
-2\int_\Omega \frac{\alpha_t}{\alpha}\bar{p}_{ttt} \Delta \bar{p}_{ttt}\, dx
-\int_\Omega \frac{\alpha_{tt}}{\alpha}\bar{p}_{tt} \Delta \bar{p}_{ttt}\, dx
\\
&+\frac{1}{2} c^2 \frac{d}{dt} \left\Vert \frac{1}{\sqrt{\alpha}} \Delta \bar{p}_{tt}\right\Vert^2 
+\frac{1}{2} c^2 \int_\Omega \frac{\alpha_t}{\alpha^2} |\Delta \bar{p}_{tt}|^2\, dx
+b \left\Vert \frac{1}{\sqrt{\alpha}} \Delta \bar{p}_{ttt} \right\Vert^{2}
\end{aligned}
\]
where
\[
\begin{aligned}
-\int_{\partial\Omega} \bar{p}_{tttt}\Dnu{\bar{p}_{ttt}}\, dS
=& 
\frac{1}{c} \Bigl(\Vert\bar{p}_{tttt}\Vert_{L^2(\Gamma_a)}^2
-\int_{\Gamma_a}\bar{p}_{tttt}a_{ttt}\, dS\Bigr)
-\int_{\Gamma_D} g_{D\,tttt}\Dnu{\bar{p}_{ttt}}\, dS\\
&-\frac{d}{dt}\int_{\Gamma_N} \bar{p}_{ttt}g_{N\,ttt}\, dS
+\int_{\Gamma_N} \bar{p}_{ttt}g_{N\,tttt}\, dS
\end{aligned}
\]

Thus, integrating with respect to time we obtain
\begin{equation}\label{enest_pbarAppendix}
\begin{aligned}
&\frac12\Vert \nabla \bar{p}_{ttt} \Vert_{L^\infty(0,t;L^2(\Omega))}^{2} 
+ \frac12 c^2 \left\Vert \frac{1}{\sqrt{\alpha}} \Delta \bar{p}_{tt} \right\Vert_{L^\infty(0,t;L^2(\Omega))}^{2}
+ b \left\Vert \frac{1}{\sqrt{\alpha}} \Delta \bar{p}_{ttt} \right\Vert_{L^2(0,t;L^2(\Omega))}^{2}
+\frac{1}{c} \Vert \bar{p}_{tttt} \Vert_{L^2(0,t;L^2(\Gamma_a))}^{2}\\
&\leq
\frac{1}{c} \Vert \bar{p}_{tttt} \Vert_{L^2(0,t;L^2(\Gamma_a))}
\Vert a_{ttt} \Vert_{L^2(0,t;L^2(\Gamma_a))}
+\Vert \Dnu{\bar{p}_{ttt}}\Vert_{{L^\infty}(0,t;H^{-1/2}(\Gamma_D))} 
\Vert g_{D\,tttt}\Vert_{{L^1}(0,t;H^{1/2}(\Gamma_D))}
\\
&\quad 
+\Vert \bar{p}_{ttt}\Vert_{L^\infty(0,t;H^{1/2}(\Gamma_N))} 
\Bigl(\Vert g_{N\,ttt}\Vert_{L^\infty(0,t;H^{-1/2}(\Gamma_N))}
+\Vert g_{N\,tttt}\Vert_{L^1(0,t;H^{-1/2}(\Gamma_N))}\Bigr)
\\
&\quad
+c^2\tfrac{1}{\underline{\alpha}}
\Vert\tfrac{\alpha_{t}}{\alpha}\Vert_{L^1(\zeroT L^\infty(\Omega))}
\Vert\Delta \bar{p}_{tt}\Vert_{L^\infty(0,t;L^2(\Omega))}^2
+2 \Vert\tfrac{\alpha_{t}}{\alpha}\Vert_{L^2(\zeroT L^3(\Omega))}
\Vert\bar{p}_{ttt}\Vert_{L^\infty(0,t;L^6(\Omega))}
\Vert\Delta \bar{p}_{ttt}\Vert_{L^2(0,t;L^2(\Omega))}
\\
&\quad
+\Vert\tfrac{\alpha_{tt}}{\alpha}\Vert_{L^2(\zeroT L^3(\Omega))}
\Vert\bar{p}_{tt}\Vert_{L^\infty(0,t;L^6(\Omega))}^2
\Vert\Delta \bar{p}_{ttt}\Vert_{L^2(0,t;L^2(\Omega))}^2
\end{aligned}
\end{equation}
with $\tfrac{1}{\underline{\alpha}}=\left\Vert\frac{1}{\alpha}\right\Vert_{L^\infty(\zeroT L^\infty(\Omega))}$.

We also test with $\frac{1}{\alpha}\bar{p}_{ttt}$, which yields
\[
\begin{aligned}
0=& \frac12 \frac{d}{dt} \Vert \bar{p}_{ttt} \Vert^{2} 
-2\int_\Omega \frac{\alpha_t}{\alpha}\bar{p}_{ttt}^2 \, dx
-\int_\Omega \frac{\alpha_{tt}}{\alpha}\bar{p}_{tt} \bar{p}_{ttt}\, dx
\\&\quad
+\frac{1}{2} c^2 \frac{d}{dt} \left\Vert \frac{1}{\sqrt{\alpha}} \nabla \bar{p}_{tt}\right\Vert^2 
+\frac{1}{2} c^2 \int_\Omega \frac{\alpha_t}{\alpha^2} |\nabla \bar{p}_{tt}|^2\, dx
+b \left\Vert \frac{1}{\sqrt{\alpha}} \nabla \bar{p}_{ttt} \right\Vert^{2}\\
&-\int_\Omega (c^2\nabla{\bar{p}_{tt}}+b\nabla{\bar{p}_{ttt}})\cdot\nabla\alpha\,
\tfrac{1}{\alpha^2}p_{ttt}
-\int_{\partial\Omega} (c^2\Dnu{\bar{p}_{tt}}+b\Dnu{\bar{p}_{ttt}})\tfrac{1}{\alpha}p_{ttt} \, dS,
\end{aligned}
\]
where 
\[
\begin{aligned}
&\int_{\partial\Omega} (c^2\Dnu{\bar{p}_{tt}}+b\Dnu{\bar{p}_{ttt}})\tfrac{1}{\alpha}p_{ttt} \, dS
= c \Bigl(\Vert\tfrac{1}{\sqrt{\alpha}}\bar{p}_{ttt}\Vert_{L^2(\Gamma_a)}^2
-\int_{\Gamma_a}\tfrac{1}{\alpha}\bar{p}_{ttt}a_{tt}\, dS\Bigr)\\
&\quad+\frac{b}{c}\Bigl( \frac12\frac{d}{dt} \Vert\tfrac{1}{\sqrt{\alpha}}\bar{p}_{ttt}\Vert_{L^2(\Gamma_a)}^2
+\int_{\Gamma_a}\bigl(\tfrac{\alpha_t}{\alpha^2} |\bar{p}_{ttt}|^2
-\tfrac{1}{\alpha}\bar{p}_{ttt}a_{ttt}\bigr)\, dS\Bigr)\\
&
-\int_{\Gamma_D} (c^2\Dnu{\bar{p}_{tt}}+b\Dnu{\bar{p}_{ttt}})\tfrac{1}{\alpha}g_{D\,ttt}\, dS
-\int_{\Gamma_N} (c^2g_{N\,tt}+bg_{N\,ttt})\tfrac{1}{\alpha}\bar{p}_{ttt}\, dS,
\end{aligned}
\]
thus, after integration with respect to time,
\begin{equation}\label{enest_pbar0}
\begin{aligned}
&\frac12(1-2\bar{\alpha})\Vert \bar{p}_{ttt} \Vert_{L^\infty(0,t;L^2(\Omega))}^{2} 
+ \frac12 c^2(1-\bar{\alpha}) \left\Vert \frac{1}{\sqrt{\alpha}} \nabla \bar{p}_{tt} \right\Vert_{L^\infty(0,t;L^2(\Omega))}^{2}
+ b \left\Vert \frac{1}{\sqrt{\alpha}} \nabla \bar{p}_{ttt} \right\Vert_{L^2(0,t;L^2(\Omega))}^{2}\\
&\hspace*{5cm}+ c\Vert \frac{1}{\sqrt{\alpha}}\bar{p}_{ttt} \Vert_{L^2(0,t;L^2(\Gamma_a))}^{2}
+ \frac{b}{c}(1-\bar{\alpha})\Vert \frac{1}{\sqrt{\alpha}}\bar{p}_{ttt} \Vert_{L^\infty(0,t;L^2(\Gamma_a))}^{2}
\\
&\leq
\Vert\tfrac{\alpha_{tt}}{\alpha}\Vert_{L^2(\zeroT L^2(\Omega))}
\Vert\bar{p}_{tt}\Vert_{L^\infty(0,t;L^4(\Omega))}
\Vert\bar{p}_{ttt}\Vert_{L^2(0,t;L^4(\Omega))}
\\&\quad
+\tfrac{1}{\underline{\alpha}}\Vert\tfrac{\nabla\alpha}{\alpha}\Vert_{L^2(0,t;L^3(\Omega)}
\Vert\bar{p}_{ttt}\Vert_{L^\infty(0,t;L^6(\Omega))}
\Bigl(c^2\Vert\nabla\bar{p}_{tt}\Vert_{L^2(0,t;L^2(\Omega))}
+b\Vert\nabla\bar{p}_{ttt}\Vert_{L^2(0,t;L^2(\Omega))}
\Bigr)
\\&\quad
+c \Vert\tfrac{1}{\sqrt{\alpha}} \bar{p}_{ttt} \Vert_{L^2(0,t;L^2(\Gamma_a))}
\Vert \tfrac{1}{\sqrt{\alpha}} a_{tt} \Vert_{L^2(0,t;L^2(\Gamma_a))}
+\tfrac{b}{c} \Vert\tfrac{1}{\sqrt{\alpha}} \bar{p}_{ttt} \Vert_{L^2(0,t;L^2(\Gamma_a))}
\Vert \tfrac{1}{\sqrt{\alpha}} a_{ttt} \Vert_{L^2(0,t;L^2(\Gamma_a))}
\\&\quad
+\tfrac{1}{\underline{\alpha}}\Bigl(
{
c^2\Vert \Dnu{\bar{p}_{tt}}\Vert_{L^\infty(0,t;H^{-1/2}(\Gamma_D))} 
\Vert g_{D\,ttt}\Vert_{L^1(0,t;H^{1/2}(\Gamma_D))}
+b\Vert \Dnu{\bar{p}_{ttt}}\Vert_{L^\infty(0,t;H^{-1/2}(\Gamma_D))} 
\Vert g_{D\,ttt}\Vert_{L^1(0,t;H^{1/2}(\Gamma_D))}
}
\Bigr)
\\
&\quad 
+\tfrac{1}{\underline{\alpha}}\Vert \bar{p}_{ttt}\Vert_{L^\infty(0,t;H^{1/2}(\Gamma_N))} 
\Bigl(c^2\Vert g_{N\,tt}\Vert_{L^1(0,t;H^{-1/2}(\Gamma_N))}
+b\Vert g_{N\,ttt}\Vert_{L^1(0,t;H^{-1/2}(\Gamma_N))}\Bigr)
\end{aligned}
\end{equation}

{
In case $\Gamma_D\not=\emptyset$ we can estimate 
\[
\Vert \Dnu{\bar{p}_{ttt}}\Vert_{L^\infty(0,t;H^{-1/2}(\Gamma_D))}
\leq C^{tr}_{-1/2}
\Vert \bar{p}_{ttt}\Vert_{L^\infty(0,t;H^1)}
\leq C^{tr}_{-1/2} C^\Omega_{\Gamma_D}
\Bigl(\Vert \nabla \bar{p}_{ttt} \Vert_{L^\infty(0,t;L^2(\Omega))}^2+\Vert g_{D\,ttt} \Vert_{L^\infty(0,t;L^2(\Gamma_D))}^2\Bigr)^{1/2} 
\]
using \eqref{PF} and employing the higher order energy estimate \eqref{enest_pbarAppendix}.
\\ 
If $\Gamma_D=\emptyset$ an estimate of $\Vert \Dnu{\bar{p}_{ttt}}\Vert_{L^\infty(0,t;H^{-1/2}(\Gamma_D))}$
is anyway not needed and actually the lower order energy estimate \eqref{enest_pbar0} suffices to obtain $p_{ttt}\in L^1(\zeroT L^2(\Omega))$, which suffices for our purposes.
}

We {combine \eqref{enest_pbarAppendix}, \eqref{enest_pbar0},} assume $\bar{\alpha}=\Vert\tfrac{\alpha_{t}}{\alpha}\Vert_{L^1(\zeroT L^\infty(\Omega))}$  to be small enough and apply Young's inequality as well as continuity of the embedding $H^1(\Omega)\to L^6(\Omega)$, to arrive at the energy estimate
\[
\begin{aligned}
&\Vert \bar{p}_{ttt} \Vert_{L^\infty(0,t;H^1(\Omega))}^{2} 
+ \left\Vert \frac{1}{\sqrt{\alpha}} \Delta \bar{p}_{tt} \right\Vert_{L^\infty(0,t;L^2(\Omega))}^{2}
+ b\left\Vert \frac{1}{\sqrt{\alpha}} \Delta \bar{p}_{ttt} \right\Vert_{L^2(0,t;L^2(\Omega))}^{2}
+ \Vert \bar{p}_{tttt} \Vert_{L^2(0,t;L^2(\Gamma_a))}^{2}
\\
&\leq C \Bigl(\Vert a_{ttt} \Vert_{L^2(0,t;L^2(\Gamma_a))}^2
{+\Vert g_{D\,ttt}\Vert_{L^\infty(\zeroT H^{1/2}(\Gamma_D))}^2
+\Vert g_{D\,tttt}\Vert_{L^1(\zeroT H^{1/2}(\Gamma_D))}^2}
+\Vert g_{N\,ttt}\Vert_{W^{1,1}(0,t;H^{-1/2}(\Gamma_N))}^2
\Bigr)
\end{aligned}
\]

Analogous estimates obtained by testing the original equation and its time differentiated version with $-\frac{1}{\alpha}\Delta \bar{p}_{t}$ and $-\frac{1}{\alpha}\Delta \bar{p}_{tt}$, respectively, yield 
\begin{equation}\label{enest_pbar}
\Ep_1[\bar{p}](t)\leq C\Bigl(\Ep_1[\bar{p}](0)
+ G[a,g_D,g_N]+ G[a_{t},g_{D\,t},g_{N\,t}]
\Bigr).
\end{equation}

Likewise, with two equipartition of energy estimates obtained by testing the original equation with $-\frac{1}{\alpha}\Delta \bar{p}$ and its time differentiated version by $-\frac{1}{\alpha}\Delta \bar{p}_{t}$, we obtain \eqref{enest_eqp_pbar}.  

\end{proof}

\end{document}